\documentclass[11pt,a4paper,reqno,intlimits,sumlimits]{amsart}

\usepackage{amssymb}
\usepackage{color}
\usepackage[mathscr]{euscript}
\usepackage{xspace}
\usepackage{enumerate}
\usepackage{epsfig,rotating}
\usepackage{graphicx}
\input{xy}\xyoption{all}
\usepackage{todonotes}
\usepackage{amsmath}
\usepackage{lscape}
\usepackage[latin1]{inputenc}
\usepackage{chicago}
\usepackage{accents}
\usepackage{dsfont}
\usepackage{comment}
\usepackage{mathrsfs}
\usepackage{float}
\usepackage{footmisc}
\usepackage{ragged2e}
\usepackage{multirow}
\usepackage{multicol}

\usepackage[left=1.42in,right=1.42in,top=1.6in,bottom=1.3in]{geometry}

\DeclareMathAlphabet{\mathpzc}{OT1}{pzc}{m}{it}

\usepackage[bookmarksopen,pdfstartview=FitH]{hyperref}
\hypersetup{colorlinks,%
            citecolor=blue,%
            filecolor=blue,%
            linkcolor=blue,%
            urlcolor=blue}%

\theoremstyle{plain}
\newtheorem{theorem}{Theorem}[section]
\newtheorem{lemma}[theorem]{Lemma}
\newtheorem{proposition}[theorem]{Proposition}

\theoremstyle{definition}
\newtheorem{remark}[theorem]{Remark}

\newtheorem{definition}[theorem]{Definition}

\newtheorem{assumption}{Assumption}

\numberwithin{equation}{section}

\numberwithin{equation}{section}

\makeatletter
\def\Ddots{\mathinner{\mkern1mu\raise\p@
\vbox{\kern7\p@\hbox{.}}\mkern2mu
\raise4\p@\hbox{.}\mkern2mu\raise7\p@\hbox{.}\mkern1mu}}
\makeatother

\newcommand{\de}[1]{\partial_{#1}}

\def\F{\ensuremath{\mathcal{F}}}

\def\R{\ensuremath{\mathbb{R}}}

\def\C{\ensuremath{\mathbb{C}}}

\def\N{\mathbb{N}}

\def\B{\mathcal{B}}
\def\X{\mathcal{X}}
\def\H{\mathcal{H}}

\def\G{\mathcal{G}}

\def\Y{\mathcal{Y}}

\def\Pc{\mathcal{P}}

\newcommand{\var}{\mathbb{V}\text{ar}}
\def\P{\mathbb{P}}
\def\Q{\mathbb{Q}}
\def\E{\ensuremath{\mathrm{I\kern-.2em E}}}

\def\hP{\ensuremath{\widehat{\mathrm{I\kern-.2em P}}}}
\def\hE{\ensuremath{\widehat{\mathrm{I\kern-.2em E}}}}

\def\bP{\ensuremath{\overline{\mathrm{I\kern-.2em P}}}}
\def\bE{\ensuremath{\overline{\mathrm{I\kern-.2em E}}}}

\def\y1{y_{-1}}

\def\e{\mathrm{e}}

\newcommand*{\LargerCdot}{\raisebox{-0.25ex}{\scalebox{1.2}{$\cdot$}}}
\newcommand{\q}[1]{\hspace*{-.7mm}\left\langle #1 \right\rangle}
\newcommand{\p}[1]{\overset{\LargerCdot}{#1}}
\newcommand{\px}[1]{\p{x_{#1}}^{\underset{\hspace{-.1cm}ac}{}}}

\makeatletter
\allowdisplaybreaks

\begin{document}

\begin{center}
{\Large \bf Long-time trajectorial large deviations for affine stochastic volatility models and application to variance reduction for option pricing}
\vspace*{.5cm}
\begin{multicols}{3}
Zorana {\sc Grbac}\footnote{\label{fn:Paris7}Universit\'e Paris Diderot, Paris, France} \\
David {\sc Krief} \footref{fn:Paris7}\\
Peter {\sc Tankov} \footref{fn:Paris7}\hspace{.03cm}\footnote{ENSAE, Palaiseau, France}
\end{multicols} 
\vspace*{.5cm}
\end{center}
{\justifying {\bf Abstract:} {\it This work extends the variance reduction method for the pricing of possibly path-dependent derivatives, which was developed in \cite{Ge:Ta} for exponential L\'evy models, to affine stochastic volatility models \cite{Kel2011}. We begin by proving a pathwise large deviations principle for affine stochastic volatility models. We then apply a time-dependent Esscher transform to the affine process and use Varadhan's Lemma, in the fashion of \cite{Gu:Ro2008} and \cite{Rob2010}, to approximate the problem of finding the Esscher measure that minimises the variance of the Monte-Carlo estimator. We test the method on the Heston model with and without jumps to demonstrate the numerical efficiency of the method.}}

\begin{center}
\end{center}
\setlength{\parindent}{0in}

\section{Introduction}
The aim of this paper is to develop efficient importance sampling
estimators for prices of path-dependent options in affine
stochastic volatility (ASV) models of asset prices. To this end, we
establish pathwise large deviation results for these models, which are
of independent interest. 

An ASV model, studied in \cite{Kel2011} is a
two-dimensional affine process $(X,V)$ on $\mathbb R\times
\mathbb R_+$ with special properties, where $X$ models the logarithm
of the stock price and $V$ its instantaneous variance.  This class includes many well studied and widely used models such as Heston
stochastic volatility model \cite{Hes1993}, the model of Bates \cite{Bat1996},
Barndorff-Nielsen stochastic volatility model \cite{Ba:Sh2001} and time-changed
L\'evy models with independent affine time change. European options in affine
stochastic volatility models may be priced by Fourier transform, but
for path-dependent options explicit formulas are in general not
available and Monte Carlo is often the method of choice. At the same
time, Monte Carlo simulation of such processes is difficult and
time-consuming: the convergence rates of
discretization schemes are often low due to the irregular nature of
coefficients of the corresponding stochastic differential
equations. To accelerate Monte Carlo simulation, it is thus
important to develop efficient variance-reduction algorithms for these
models. 

In this paper, we therefore develop an importance sampling algorithm
for ASV models. 
The importance sampling method is based on the following identity, valid for any probability measure $\mathbb Q$, with respect to which $\mathbb P$ is absolutely continuous. Let $P$ be a deterministic function of a random trajectory $S$, then
$$
\mathbb E[P(S)] = \mathbb E^{\mathbb Q}\left[\frac{d\P}{d\Q}P(S)\right].
$$
This allows one to define the {importance sampling estimator}
$$
\widehat P^{\mathbb Q}_N := \frac{1}{N} \sum_{j=1}^N \left[\frac{d\mathbb
  P}{d\mathbb Q}\right]^{(j)} P(S^{(j)}_{\mathbb Q}),
$$
where $S^{(j)}_{\mathbb Q}$ are i.i.d.~sample trajectories of $S$ under the
measure $\mathbb Q$. For efficient variance reduction, one needs then to find a probability measure $\mathbb Q$
such that $S$ is {easy to simulate} under $\mathbb Q$ and the variance 
$$
\var_{\mathbb Q} \left[ P(S)\frac{d\mathbb P}{d\mathbb Q} \right] 
$$ 
is considerably smaller than the original variance $\var_{\mathbb P} \left[P(S) \right]$. 
 
In this paper, following the work of \cite{Ge:Ta} in the context of
L\'evy processes, we define the probability $\mathbb Q$ using the path-dependent Esscher transform,
$$
\frac{d\P_\theta}{d\mathbb P} = \frac{e^{\int_{[0,T]} X_t
    \cdot \theta(dt)}}{\mathbb E\left[e^{\int_{[0,T]} X_t
    \cdot \theta(dt)}\right]},
$$
where $X$ is the first component of the ASV
model (the logarithm of stock price) and $\theta$ is a (deterministic) bounded 
signed measure on $[0,T]$. The optimal choice of $\theta$ should minimize the variance of the estimator under $\mathbb P_\theta$,
$$
\var_{\mathbb P_\theta} \left( P(S) \frac{d\mathbb P}{d\P_\theta}\right) = \mathbb E_{\mathbb P}\left[P^2(S)\frac{d\mathbb P}{d\P_\theta}\right] - \mathbb E\left[P(S)\right]^2.
$$

The computation of this variance is in general as difficult as the
computation of the option price itself. Following
\cite{Du:Wa2004,Gl:He:Sh1999,Gu:Ro2008,Rob2010} and more recently \cite{Ge:Ta}, we propose to compute the variance
reduction measure $\theta^*$ by minimizing the \emph{proxy} for the
variance computed using the theory of large deviations. 

To this end, we establish a pathwise large deviation principle (LDP) for affine stochastic volatility models. A one dimensional LDP for $X_t/t$ as $t\to \infty$
where $X$ is the first component of an ASV model has been proven in \cite{Ja:Ke:Mi2013}. In this paper, we extend this result to the trajectorial setting, in the spirit of the pathwise LDP principles of \cite{Leo2000}, but in a weaker topology. 

The rest of the paper is structured as follows. In Section \ref{sec:Model Description}, we describe the model and recall certain useful properties of ASV processes. In Section \ref{sec:LargeDeviationsTheory}, we recall some general results of large deviations theory. In Section \ref{sec:LDP-Affine}, we prove a LDP for the trajectories of ASV processes. In Section \ref{sec:VarianceReduction}, we develop the variance reduction method, using an asymptotically optimal change of measure obtained via the LDP shown in Section \ref{sec:LDP-Affine}. In Section \ref{sec:NumericalExamples}, we test the method numerically on several examples of options, some of which are path-dependent, in the Heston model with and without jumps.

\section{Model description} \label{sec:Model Description}
In this paper, we model the price of the underlying asset $(S_t)_{t \ge 0}$ of an option as $S_t = S_0 \, e^{X_t}$, where we model $(X_t)_{t \ge 0}$ as an affine stochastic volatility process. We recall, from \cite{Kel2011} and \cite{Du:Fi:Sc2003}, the definition and some properties of ASV models.
\begin{definition}
An ASV model $(X_t,V_t)_{t \ge 0}$, is a stochastically continuous, time-homogeneous Markov process such that $\left(e^{X_t}\right)_{t \ge 0}$ is a martingale and
\begin{equation}\label{eq:LaplaceTransform}
\E\left( e^{u X_{t} + w V_{t}}\middle|X_0=x,V_0=v\right) = e^{\phi(t,u,w) + \psi(t,u,w) \, v + u \, x}\:,
\end{equation}
for all $(t,u,w) \in \R_+ \!\!\times \C^2$.
\end{definition}
\begin{proposition}
The functions $\phi$ and $\psi$ satisfy generalized Riccati equations
\begin{subequations} \label{eq:GeneralizedRiccati}
\begin{align}
\de{t} \phi(t,u,w) & = F(u,\psi(t,u,w))  \:, && \phi(0,u,w) = 0 \label{eq:GeneralizedRiccati-a}\\
\de{t} \psi(t,u,w) & = R(u,\psi(t,u,w))  \:, && \psi(0,u,w) = w \label{eq:GeneralizedRiccati-b}\:,
\end{align}
\end{subequations}
where $F$ and $R$ have the L\'evy-Khintchine forms
\begin{align*}
F(u,w) & = 
\left(\begin{array}{cc}
u & w 
\end{array}\right) 
\cdot \frac{a}{2} \cdot 
\left(\begin{array}{c}
u \\
w
\end{array}\right) 
+ b \cdot
\left(\begin{array}{c}
u \\
w
\end{array}\right) \\
& \qquad + \int_{D\backslash\{0\}}\left(e^{xu+yw}-1-w_F(x,y) \cdot \left(\begin{array}{c}
u \\
w
\end{array}\right)\right) m(dx,dy) \:, \\
R(u,w) & = 
\left(\begin{array}{cc}
u & w 
\end{array}\right) 
\cdot \frac{\alpha}{2} \cdot 
\left(\begin{array}{c}
u \\
w
\end{array}\right) 
+ \beta \cdot
\left(\begin{array}{c}
u \\
w
\end{array}\right) \\
& \qquad + \int_{D\backslash\{0\}}\left(e^{xu+yw}-1-w_R(x,y) \cdot \left(\begin{array}{c}
u \\
w
\end{array}\right)\right) \mu(dx,dy) \:,
\end{align*}
where $D = \R \times \R_+$, 
\[
w_F(x,y) = 
\left(\begin{array}{c}
\frac{x}{1+x^2} \\
0
\end{array}\right) 
\qquad \text{and} \qquad
w_R(x,y) = 
\left(\begin{array}{c}
\frac{x}{1+x^2} \\
\frac{y}{1+y^2} 
\end{array}\right)
\]
and $(a,\alpha, b, \beta, m, \mu)$ satisfy the following conditions
\begin{itemize}
\item $a,\alpha$ are positive semi-definite 2$\times$2-matrices where $a_{12}=a_{21}=a_{22}=0$.
\item $b \in D$ and $\beta \in \R^2$.
\item $m$ and $\mu$ are L\'evy measures on $D$ and $\int_{D\backslash \{0\}} ((x^2+y) \wedge 1) \, m(dx,dy) < \infty$.
\end{itemize}
\end{proposition} 
In the rest of the paper, we assume that there exists $u \in \R$ such that $R(u,0) \neq 0$, for the law of $(X_t)_{t \ge 0}$ to depend on $V_0$.
Define the function 
\[
\chi(u) = \left. \de{w} R(u,w) \right|_{w=0} = \alpha_{12} u + \beta_2 + \int_{D\backslash\{0\}} y \left(e^{xu}- \frac{1}{1+y^2}\right) \mu(dx,dy) \:.
\] 
A sufficient condition for $S_t = S_0 \, e^{X_t}$ to be a martingale \cite[Corollary 2.7]{Kel2011}, which we assume to be satisfied in the sequel, is $F(1,0)=R(1,0)=0$ and $\chi(0) + \chi(1) < \infty$. 

In the following theorem, we compile several results of \cite{Kel2011} that describe the behaviour of the solution to eq. \eqref{eq:GeneralizedRiccati} as $t \rightarrow \infty$. 
\begin{theorem}\label{thm:KellerResselProperties}
Assume that $\chi(0) < 0$ and $\chi(1) < 0$. 
\begin{itemize}
\item There exists an interval $I \supseteq [0,1]$, such that for each $u \in I$, eq. \eqref{eq:GeneralizedRiccati-b} admits a unique stable equilibrium $w(u)$. 
\item For $u \in I$, eq. \eqref{eq:GeneralizedRiccati-b} admits at most one other equilibrium $\tilde{w}(u)$, which is unstable. 
\item For $u \in \R \backslash I$, eq. \eqref{eq:GeneralizedRiccati-b} does not have any equilibrium. 
\end{itemize}
We denote $\mathcal{B}(u)$ the basin of attraction of the stable solution $w(u)$ of eq. \eqref{eq:GeneralizedRiccati-b} and $J=\{u\in I\,:\,F(u,w(u)) < \infty\}$, the domain of $u \mapsto F(u,w(u))$. We have that
\begin{itemize}
\item $J$ is an interval such that $[0,1] \subseteq J \subseteq I$.
\item For $u \in I$, $w \in \B(u)$ and $\Delta t > 0$, we have 
\begin{equation}
\psi\left(\frac{\Delta t}{\epsilon},\, u,\, w \right) \underset{\epsilon \rightarrow 0}{\longrightarrow} w(u) \:. \label{eq:ConvergencePsi}
\end{equation}
\item For $u \in J$, $w \in \B(u)$ and $\Delta t > 0$,
\begin{equation}
\epsilon \, \phi\left(\frac{\Delta t}{\epsilon},\, u,\, w \right) \underset{\epsilon \rightarrow 0}{\longrightarrow} \Delta t \, h\left(u\right) \label{eq:ConvergencePhi}\:,
\end{equation}
where $h(u) = F(u,w(u)) = \lim_{\epsilon \rightarrow 0} \epsilon \log \E\left[e^{u X_{1/\epsilon}}\right]$. 
\item For every $u \in I$, $0 \in \B(u)$.
\end{itemize}
\end{theorem}
\begin{definition}
A convex function $f: \R^n \rightarrow \R \cup \{\infty\}$ with
effective domain $D_f$ is \emph{essentially smooth} if 
\begin{itemize}
\item[i.] $D^\circ_f$ is non-empty;
\item[ii.] $f$ is differentiable in $D^\circ_f$;
\item[iii.] $f$ is \emph{steep}, that is, for any sequence $(u_n)_{n\in\N} \subset D_f^\circ$ that converges to a point in the boundary of $D_f$, 
\[
\lim_{n \rightarrow \infty} ||\nabla f(u_n)|| = \infty \:.
\]
\end{itemize}
\end{definition}

In the rest of the paper, we shall make the following assumptions on the model.
\begin{assumption}\label{ass:H}
The function $h$ satisfies the following properties. 
\begin{enumerate}
\item\label{ass:Ha} There exists $u < 0$, such that $h(u) < \infty$. 
\item\label{ass:Hb} $u \mapsto h(u)$ is essentially smooth.
\end{enumerate}
\end{assumption}
In \cite{Ja:Ke:Mi2013}, a set of sufficient conditions is provided for Assumption \ref{ass:H} to be verified:
\begin{proposition}[Corollary 8 in \cite{Ja:Ke:Mi2013}]
Let $(X, V)$ be an ASV model such that $u \mapsto R(u,0)$ and $w \mapsto F(0,w)$ are not identically 0 and $\chi(0)$ and $\chi(1)$ are strictly negative. If either of the following conditions holds
\begin{enumerate}[(i)]
\item The L\'evy measure $\mu$ of $R$ has exponential moments of all orders, F is steep and $(0,0),(1,0) \in D_F^\circ$.
\item $(X,V)$ is a diffusion,
\end{enumerate}
then function $h$ is well defined, for every $u \in \mathbb R$ with effective domain $J$. Moreover h is essentially smooth and $\{0, 1\} \subset J^\circ$. 
\end{proposition}
We now discuss the form of the basin of attraction of the unique stable solution of \eqref{eq:GeneralizedRiccati-b}.
\begin{lemma}{\cite[Lemma 2.2.]{Kel2011}}\label{lem:RConvex}
\begin{enumerate}[(a)]
\item $F$ and $R$ are proper closed convex functions on $\R^2$.
\item $F$ and $R$ are analytic in the interior of their effective domain.
\item Let $U$ be a one-dimensional affine subspace of $R^2$. Then $F|U$ is either a strictly convex or an affine function. The same holds for $R|U$.
\item  If $(u, w) \in D_F$, then also $(u,\eta) \in D_F$ for all $\eta \le w$. The same holds for $R$.
\end{enumerate}
\end{lemma}
\begin{lemma}\label{lem:BasinOfAttraction}
Let $f : \R \rightarrow \R \cup \{+\infty\}$ be a convex function with either two zeros $w < \tilde{w}$, or a single zero $w$. In the latter case, we let $\tilde{w} = \infty$. Assume that there exists $y \in (w,\tilde{w})$ such that $f(y) < 0$. Then for every $x \in D_f$,
\[
\begin{cases}
f(x) > 0 \:, \qquad & \text{if} \:\: x < w \quad \text{or} \quad \tilde{w} < x \:, \\
f(x) < 0 \:, \qquad & \text{if} \:\: x \in (w,\tilde{w}) \:.
\end{cases} 
\]
\end{lemma}
\begin{proof}
By convexity, for every $x \in D_f$ such that $x < w$,
\[
\frac{y-w}{y-x} \, f(x) + \frac{w-x}{y-x} \, f(y) \ge f(w) = 0 
\]
and therefore $f(x) \ge -\frac{w-x}{y-w} \, f(y) > 0$. Furthermore, for every $x \in (w,y]$, 
\[
f(x) \le \frac{y-x}{y-w} \, f(w) + \frac{x-w}{y-w} \, f(y) < 0 \:.
\]
Let $s = sup\{x \in D_f \::\: f(x) < 0 \}$. If $f$ is continuous in $s$, then $\tilde{w} = s$ and for every $x > \tilde{w}$ in $D_f$, $f(x) \ge -\frac{\tilde{w}-x}{y-\tilde{w}} \, f(y) > 0$. If $f$ is discontinuous in $s$ however, then by convexity, $f(x) = +\infty$ for $x > s$.
\end{proof}
\begin{proposition}
Let $u \in I$ and consider $w(u)$ the stable equilibrium of \eqref{eq:GeneralizedRiccati-b}. Then the basin of attraction of $w(u)$ is $\B(u) = (-\infty, \tilde{w}(u)) \cap D_{R(u,\cdot)}$, where $\tilde{w}(u) = \infty$ when \eqref{eq:GeneralizedRiccati-b} admits only one equilibrium.
\end{proposition}
\begin{proof}
By Lemma \ref{lem:RConvex}, $w \mapsto R(u,w)$ is convex. Since $w(u)$ is a stable equilibrium, the hypotheses of Lemma \ref{lem:BasinOfAttraction} are verified. Therefore, $R(u, w) > 0$ for every $w < w(u)$, whereas $R(u, w) < 0$ for every $w \in D_{R(u,\cdot)}$ such that $ w(u) < w < \tilde{w}(u)$. This implies that the solution of 
\begin{equation}\label{eq:REquationAlone}
\de{t} \psi(t,u,w) = R(u,\psi(t,u,w))  \:, \qquad \psi(0,u,w) = w
\end{equation}
converges to $w(u)$ for every $w \in (-\infty, \tilde{w}(u)) \cap D_{R(u,\cdot)}$, whereas, if $w > \tilde{w}$, the solution of \eqref{eq:REquationAlone} diverges to $\infty$.
\end{proof}

\section{Large deviations theory} \label{sec:LargeDeviationsTheory}

In this section, we recall some useful classical results of the large deviations theory. We refer the reader to \cite{De:Ze} for the proofs and for a broader overview of the theory.
\begin{theorem}[G\"artner-Ellis]\label{thm:GartnerEllis}
Let $\left(X^\epsilon\right)_{\epsilon \in ]0,1]}$ be a family of random vectors in $\R^n$ with associated measure $\mu_\epsilon$. Assume that for each $\lambda \in \R^n$,
\[
\Lambda(\lambda) := \lim_{\epsilon \rightarrow 0} \epsilon \log \E\left[ e^{\frac{\q{\lambda,X^\epsilon}}{\epsilon}}\right] 
\]
as an extended real number. Assume also that 0 belongs to the interior of $D_\Lambda:=\{\theta \in \R^n \,:\, \Lambda(\theta) < \infty\}$. Denoting 
\[
\Lambda^*(x) = \sup_{\theta\in \R^n} \q{\theta,x} - \Lambda(\theta)\:, 
\]
the following hold:
\begin{enumerate}[(a)]
\item For any closed set $F$,
\[
\limsup_{\epsilon \rightarrow 0} \epsilon \, \log \mu_\epsilon(F) \le -\inf_{x \in F} \Lambda^*(x) \:.
\]
\item For any open set $G$,
\[
\liminf_{\epsilon \rightarrow 0} \epsilon \, \log \mu_\epsilon(G) \ge -\inf_{x \in G \cap \F} \Lambda^*(x) \:,
\]
where $\F$ is the set of exposed points of $\Lambda^*$, whose exposing hyperplane belongs to the interior of $D_\Lambda$.
\item If $\Lambda$ is an essentially smooth, lower semi-continuous function, then $\mu_\epsilon$ satisfies a LDP with good rate function $\Lambda^*$. 
\end{enumerate}
\end{theorem}
\begin{definition}
A partially ordered set $(\Pc, \le)$ is called \textit{right-filtering} if for every $i,j \in \Pc$, there exists $k \in \Pc$ such that $i \le k$ and $j \le k$.
\end{definition}
\begin{definition}
A \textit{projective system} $(\Y_j, p_{ij})_{i \le j \in \Pc}$ on a partially ordered right-filtering set $(\Pc,\le)$ is a family of Hausdorff topological spaces $(\Y_j)_{j\in\Pc}$ and continuous maps $p_{ij}: \Y_j \rightarrow \Y_i$ such that $p_{ik}=p_{ij}\circ p_{jk}$ whenever $i \le j \le k$.
\end{definition}
\begin{definition}
Let $(\Y_j, p_{ij})_{i \le j \in \Pc}$ be a projective system on a partially ordered right-filtering set $(\Pc,\le)$. The \textit{projective limit} of $(\Y_j, p_{ij})_{i \le j \in \Pc}$, denoted $\X = \underset{\longleftarrow}{\lim} \Y_j$, is the subset of topological spaces $\Y = \prod_{j \in \Pc} \Y_j$, consisting of all the elements $x = (y_j)_{j \in \Pc}$ for which $y_i = p_{ij}(y_j)$ whenever $i \le j$, equipped with the topology induced by $\Y$. The \textit{projective limit of closed subsets} $F_j \subseteq \Y_j$ are defined in the same way and denoted $F = \underset{\longleftarrow}{\lim} F_j$.
\end{definition}
\begin{remark}
The canonical projections of $\X$, i.e. the restrictions $p_j: \X \rightarrow \Y_j$ of the coordinate maps from $\X$ to $\Y_j$, are continuous. 
\end{remark}
\begin{theorem}[Dawson-G\"artner]\label{thm:DawsonGartner}
Let $(\Y_j, p_{ij})_{i \le j \in \Pc}$ be a projective system on a partially ordered right-filtering set $(\Pc,\le)$ and let $(\mu_\epsilon)$ be a family of probabilities on $\X=\underset{\longleftarrow}{\lim} \Y_j$, such that for any $j \in \Pc$, the Borel probability $\mu_\epsilon \circ p_j^{-1}$ on $\Y_j$ satisfies the LDP with good rate function $\Lambda^*_j$. Then $\mu_\epsilon$ satisfies the LDP with good rate function
\[
\Lambda^*(x) = \sup_{j \in\Pc} \Lambda^*_{j}(p_j(x))\:.
\]
\end{theorem}
\begin{theorem}[Varadhan's Lemma, version of \cite{Gu:Ro2008}]\label{thm:Varadhan}
Let $(X^\epsilon)_{\epsilon \in ]0,1\,]}$ be a family of $\X$-valued random variables, whose laws $\mu_\epsilon$ satisfy a LDP with rate function $\Lambda^*$. If $\varphi: \X \rightarrow \R \cup \{-\infty\}$ is a continuous function which satisfies 
\[
\limsup_{\epsilon \rightarrow 0} \epsilon\,\log \E\left[\exp\left(\frac{\gamma\,\varphi(X^\epsilon)}{\epsilon}\right) \right] < \infty
\]
for some $\gamma > 1$, then 
\[
\lim_{\epsilon\rightarrow 0} \epsilon \,\log\E\left[ \exp\left( \frac{\varphi(X^\epsilon)}{\epsilon} \right) \right] = \sup_{x\in\X} \{ \varphi(x)-\Lambda^*(x) \}\:.
\]
\end{theorem}

\section{Trajectorial large deviations for affine\texorpdfstring{\newline}{} stochastic volatility model} \label{sec:LDP-Affine}

In this section, we prove a trajectorial LDP for $(X_t)$ when the time horizon is large. Define, for $\epsilon \in (0,1]$ and $0 \le t \le T$, the scaling $X_t^\epsilon=\epsilon X_{t/\epsilon}$. We proceed by proving first a LDP for $X_t^\epsilon$ in finite dimension, that we extend, in a second step to the whole trajectory of $(X_t^\epsilon)_{0 \le t \le T}$. 

\subsection{Finite-dimensional LDP} 

Let $\tau=\{0<t_1 < ... < t_n = t\}$, by convention $t_0=0$, and define
\[
\Lambda_{\epsilon,\tau}(\theta) = \log\E\left[e^{\sum_{k=1}^{n} \theta_k X_{t_k}^\epsilon}\right] \:,
\]
for $\theta \in \R^n$. We start by formulating our main technical assumption. 
\begin{assumption}\label{ass:H2}
One of the following conditions is verified.
\begin{enumerate}
\item\label{ass:H2a} The interval support of $F$ is $J=[u_-,u_+]$ and $w(u_-) = w(u_+)$. 
\item\label{ass:H2b} For every $u \in \R$, $\tilde{w}(\cdot) = \infty$, i.e, the generalized Riccati equations have only one (stable) equilibrium. 
\end{enumerate}
\end{assumption}
The following Lemma gives an intuition on Assumption \ref{ass:H2}. 
\begin{lemma}\label{lem:TildeW}
For every $u_1,u_2 \in I$, $\tilde{w}(u_1) \ge w(u_2)$. 
\end{lemma}
\begin{proof}
If Assumption \ref{ass:H2}(\ref{ass:H2b}) holds, then the result is obvious. Assume then that it is Assumption \ref{ass:H2}(\ref{ass:H2a}), that holds. Since $u \mapsto w(u)$ is convex and $u \mapsto\tilde{w}(u)$ is concave, then for every $u_1,u_2 \in I$,
\[
\tilde{w}(u_1) \ge \frac{u_+-u_1}{u_+-u_-} \tilde{w}(u_-) + \frac{u_1-u_-}{u_+-u_-} \tilde{w}(u_+)  = w(u_-) \:,
\]
while 
\[
w(u_2) \le \frac{u_+-u_2}{u_+-u_-} \tilde{w}(u_-) + \frac{u_2-u_-}{u_+-u_-} \tilde{w}(u_+)  = w(u_-) \:.
\]
Therefore $\tilde{w}(u_1) \ge w(u_2)$ for every $u_1,u_2 \in I$. 
\end{proof}
As a first step to apply Theorem \ref{thm:GartnerEllis}, we prove the following result.
\begin{theorem}\label{thm:LambdaTau}
Let $\theta \in \R^n$. If Assumption \ref{ass:H2} holds, then 
\[
\Lambda_{\tau}(\theta)
:= \lim_{\epsilon \rightarrow 0} \epsilon \,\Lambda_{\epsilon,\tau}(\theta/\epsilon) = 
\begin{cases}
\sum_{j=1}^n (t_{j}-t_{j-1}) \, h\left(\Theta_j\right) \quad & \text{ if } \Theta_j \in J \:, \: \forall j \\
\infty & \text{ otherwise}  
\end{cases}\:,
\]
where $\Theta_j := \sum_{k=j}^n \theta_k$.
\end{theorem}
\begin{proof}
Since Assumption \ref{ass:H2} holds, then, by Lemma \ref{lem:TildeW}, $w(\Theta_{j+1}) \in \B(\Theta_j)$ for every $j$. Assume first that $\Theta_j \in J$ for every $j$. Using the Markov property and eq. \eqref{eq:LaplaceTransform}, we obtain
\begin{align*}
\Lambda_\tau(\theta) 
& = \lim_{\epsilon \rightarrow 0} \epsilon \log \left(\E\left[e^{\sum_{j=1}^n \theta_j X_{t_j/\epsilon}}\right]\right) \\
& = \lim_{\epsilon \rightarrow 0} \epsilon \log \left(\E\left[e^{\sum_{j=1}^{n-1}\theta_j X_{t_j/\epsilon}}\, \E\left(e^{\Theta_n X_{t_n/\epsilon}}\middle| X_{t_{n-1}/\epsilon},V_{t_{n-1}/\epsilon}\right)\right]\right) \\
& = \lim_{\epsilon \rightarrow 0} \epsilon \, \phi\left(\frac{t_{n}-t_{n-1}}{\epsilon},\, \Theta_n,\, 0 \right) \\
& \qquad + \epsilon \log \left(\E\left[e^{\sum_{j=1}^{n-2} \theta_j X_{t_j/\epsilon} + \Theta_{n-1} X_{t_{n-1}/\epsilon} + \psi\left(\frac{t_{n}-t_{n-1}}{\epsilon},\, \Theta_n,\, 0 \right)V_{t_{n-1}/\epsilon}}\right]\right) \:.
\end{align*}
Since $\Theta_n \in J$ and $0 \in \B(\Theta_n)$, eqs. \eqref{eq:ConvergencePsi} and \eqref{eq:ConvergencePhi} apply and 
\begin{align*}
\Lambda_\tau(\theta) 
& = \lim_{\epsilon \rightarrow 0} \epsilon \log \left(\E\left[e^{\sum_{j=1}^{n-2} \theta_j X_{t_j/\epsilon} + \Theta_{n-1} X_{t_{n-1}/\epsilon} + \psi\left(\frac{t_{n}-t_{n-1}}{\epsilon},\, \Theta_n,\, 0 \right)V_{t_{n-1}/\epsilon}}\right]\right) \\
& \qquad + (t_{n}-t_{n-1})\, h(\Theta_n) \:.
\end{align*} 
Using the fact that $\Theta_j \in J$ and $w(\Theta_{j+1}) \in \B(\Theta_j)$ for every $j$, we can iterate the procedure to obtain
\begin{align}
\Lambda_\tau(\theta) 
& = \sum_{j=1}^n (t_{j}-t_{j-1}) \, h\left(\Theta_j\right) + \lim_{\epsilon \rightarrow 0} \epsilon \,\psi\left(\frac{t_{1}-t_{0}}{\epsilon},\, \Theta_1,\, w\left(\Theta_2 \right) \right)V_{0} + \epsilon \sum_{k=1}^n \theta_k X_0 \notag\\
& = \sum_{j=1}^n (t_{j}-t_{j-1}) \, h\left(\Theta_j\right) \label{eq:LambdaN}\:.
\end{align}

Assume now that there exists $k$ such that $\Theta_k \not\in J$. Without loss of generality, we take the largest such $k$. Following the same procedure, we find
\begin{align*}
\Lambda_\tau(\theta) 
& = \lim_{\epsilon \rightarrow 0} \epsilon \log \left(\E\left[e^{\sum_{j=1}^{k-2} \theta_j X_{t_j/\epsilon} + \Theta_{k-1} \, X_{t_{k-1}/\epsilon} + \psi\left(\frac{t_{k}-t_{k-1}}{\epsilon},\Theta_{k}, w(\Theta_{k+1})\right) \, V_{t_{k-1}/\epsilon}}\right]\right) \\
& \qquad + \epsilon \,\phi\left(\frac{t_{k}-t_{k-1}}{\epsilon},\Theta_{k}, w(\Theta_{k+1}) \right) + \sum_{j=k+1}^n (t_{j}-t_{j-1})\, h(\Theta_j) \:.
\end{align*} 
Noting that $\phi(\cdot, u, w)$ explodes in finite time for $u \not\in J$ then finishes the proof.
\end{proof}
We now proceed to the finite-dimensional large deviations result.
\begin{theorem}\label{thm:LDPFiniteDimensional}
Let $(X_t^\epsilon)_{t \ge 0,\: \epsilon \in (0,1]}$ and $\tau=\{t_1,...,t_n\}$ as previously. Assuming that Assumption \ref{ass:H2} holds, then $(X_{t_1}^\epsilon,...,X_{t_n}^\epsilon)$ satisfies a LDP on $\R^n$ with good rate function 
\[
\Lambda_\tau^*(x) = \sup_{\Theta \in J^n} \left\{ \sum_{j=1}^n \Theta_j (x_j-x_{j-1}) - \sum_{j=1}^n (t_{j}-t_{j-1}) \, h\left(\Theta_j\right) \right\} \:,
\]
where $\Theta_j = \sum_{k = j}^n \theta_k$.
\end{theorem}
\begin{proof}
By Assumption \ref{ass:H}(\ref{ass:Ha}), there exists $u \in J$ such that $u < 0$,  which implies that $[u,1] \subset J$ and therefore 0 is in the interior of $D_{\Lambda_\tau} = J^n$. Theorem \ref{thm:LambdaTau} implies that the limit 
\[
\Lambda_{\tau}(\theta)
= \lim_{\epsilon \rightarrow 0} \epsilon \,\Lambda_{\epsilon,\tau}(\theta/\epsilon) = 
\begin{cases}
\sum_{j=1}^n (t_{j}-t_{j-1}) \, h\left(\Theta_j\right) \quad & \text{ if } \Theta_j \in J \:, \: \forall j \\
\infty & \text{ otherwise}  
\end{cases}\:,
\]
where $\Theta_j := \sum_{k=j}^n \theta_k$, exists as an extended real number. Since, by Assumption \ref{ass:H}(\ref{ass:Hb}), $h$ is essentially smooth and lower semi-continuous, then so is $\Lambda_\tau$. Theorem \ref{thm:GartnerEllis} then applies and $(X_{t_1}^\epsilon,...,X_{t_n}^\epsilon)$ satisfies a LDP, on $\R^n$, with good rate function 
\[
\Lambda_\tau^*(x) = \sup_{\theta\in\R^n} \left\{\theta^\top x - \Lambda_\tau(\theta)\right\} \:.
\]
Furthermore,
\begin{align*}
\Lambda_\tau^*(x)
& = \sup_{\theta\in\R^n} \left\{\theta^\top x - \Lambda_\tau(\theta)\right\} \\
& = \sup_{\Theta \in J^n} \left\{ \sum_{j=1}^n \sum_{k=j}^n \theta_k (x_j-x_{j-1}) - \sum_{j=1}^n (t_{j}-t_{j-1}) \, h\left(\Theta_j\right) \right\} \\
& = \sup_{\Theta \in J^n} \left\{ \sum_{j=1}^n \Theta_j (x_j-x_{j-1}) - \sum_{j=1}^n (t_{j}-t_{j-1}) \, h\left(\Theta_j\right) \right\} \:,
\end{align*}
which finishes the proof.
\end{proof}

\subsection{Infinite-dimensional LDP}

\subsubsection{Extension of the LDP} 

We now extend the LDP to the whole trajectory of $(X_t^\epsilon)_{0 \le t \le T}$ on $\F([0,T], \: \R)$ $:= \{x : [0,T] \rightarrow \R, \: x_0 = 0 \}$, the set of all functions from $[0,T]$ to $\R$ that vanish at 0, by proving the following general lemma. 
\begin{lemma}\label{lem:DawsonGartner}
Assume that for any $\tau = \{t_1,...,t_n\}$, the finite-dimensional process $(X_{t_1}^\epsilon,...,X_{t_n}^\epsilon)$ satisfies a large deviation property with good rate function $\Lambda^*_\tau$. Then the family $(X_t^\epsilon)_{0\le t \le T}$ satisfies a large deviation property on $\X= \F([0,T], \: \R)$ equipped with the topology of pointwise convergence, with good rate function
\[
\Lambda^*(x) = \sup_{j \in\Pc} \Lambda^*_{j}(p_j(x))\:.
\]
\end{lemma}
\begin{proof}
Let $(\Pc,\le)$ be the partially ordered right-filtering set 
\[
\Pc= \bigcup_{n=1}^{\infty} \{(t_1,...,t_n)\: 0 \le t_1 \le ... \le t_n \le T\} 
\]
ordered by inclusion. We consider on $(\Pc,\le)$ the projective system \break$(\Y_j, p_{ij})_{i \le j \in \Pc}$ defined by $\Y_j = \R^{\#j}$ and $p_{ij}:\Y_j \rightarrow \Y_i$ the natural projection on shared times. The canonical projection from $\X$ to $\Y_\tau$ is $p_{\tau}(x) = (x_{t_1},...,x_{t_n})$. Let $\mu^\epsilon$ be the probability measure generated by $(X_t^\epsilon)_{0\le t \le T}$ on $\X$. Then, by hypothesis, for any $\tau \in \Pc$, $\mu^\epsilon \circ p_\tau^{-1}$ satisfies a LDP with good rate function $\Lambda^*_\tau$. The result is then given by Theorem \ref{thm:DawsonGartner}.
\end{proof}
\begin{theorem}\label{thm:LDPInfiniteDimensionalExistence}
Assume that Assumption \ref{ass:H2} holds, then $(X_t^\epsilon)_{0 \le t \le T}$ satisfies a LDP on $\F([0,T],\R)$ equipped with the topology of point-convergence, as $\epsilon \rightarrow 0$, with good rate function
\[
\Lambda^*(x) = \sup_\tau \Lambda_\tau^*(x)\:.
\]
\end{theorem}
\begin{proof}
The result is a direct application of Lemma \ref{lem:DawsonGartner}.
\end{proof}

\subsubsection{Calculation of the rate function}

We finally calculate the rate function of Theorem \ref{thm:LDPInfiniteDimensionalExistence}. 
\begin{theorem}\label{thm:LPDInfiniteDimensionalRateFunction}
The rate function of Theorem \ref{thm:LDPInfiniteDimensionalExistence} is 
\[
\Lambda^*(x) = \int_0^T h^*(\px{t}) \, dt + \int_0^T \H\left( \frac{d\nu_t}{d\theta_t} \right)\, d\theta_t \:,
\]
where
\[
h^*(y) = \sup_{\theta \in J} \left\{\theta y - h(\theta) \right\} \:, \qquad \H(y) = \lim_{\epsilon\rightarrow 0} \epsilon \, h^*(y/\epsilon) \:,
\]
$\px{\phantom{t}}$ is the derivative of the absolutely continuous part of $x$, $\nu_t$ is the singular component of $dx_t$ with respect to $dt$ and $\theta_t$ is any non-negative, finite, regular, $\R$-valued Borel measure, with respect to which $\nu_t$ is absolutely continuous. 
\end{theorem}
\begin{proof}
By identifying $(\Theta_1,...,\Theta_n)$ with $(\theta_{t_1},...,\theta_{t_n})$, we find for every $x \in \F([0,T],\R)$,
\begin{align*}
\sup_\tau \Lambda_\tau^*(x) & = \sup_\tau \sup_{\Theta \in J^{\# \tau}} \sum_{j=1}^{\# \tau} \Theta_j (x_{t_j}-x_{t_{j-1}}) - (t_j-t_{j-1})h(\Theta_j) \\
& =  \sup_{\theta \in \F([0,T],\R)} \sup_\tau \sum_{j=1}^{\# \tau} \theta_{t_j} (x_{t_j}-x_{t_{j-1}}) - (t_j-t_{j-1})h(\theta_{t_j}) \\
& = \sup_{\theta \in C([0,T],J)} \sup_\tau \sum_{j=1}^{\# \tau} \theta_{t_j} (x_{t_j}-x_{t_{j-1}}) - (t_j-t_{j-1})h(\theta_{t_j}) \:.
\end{align*}
Note that the supremum can be taken indifferently on $\F([0,T],J)$ or on $C([0,T],J)$ because the objective function depends on $\theta$ only on a finite set. Since we have assumed that there exists $u<0$ in $J$, then if $x$ has infinite variation, we immediately find that $\Lambda^*(x) = \infty$. Assume therefore that $x$ has finite variation. We wish to show that 
\begin{align*}
\sup_{\theta \in C([0,T],J)} \sup_\tau \sum_{j=1}^{\# \tau} \theta_{t_j} & (x_{t_j}-x_{t_{j-1}}) - (t_j-t_{j-1})h(\theta_{t_j}) \\
& = \hspace*{-.2cm} \sup_{\theta \in C([0,T],J)} \int_0^T \theta_{t} dx_t - \int_0^T h(\theta_{t}) dt \:.
\end{align*}
Notice that
\begin{align*}
& \sup_{\theta \in C([0,T],J)} \!\sup_\tau \sum_{j=1}^{\# \tau} \theta_{t_j} (x_{t_j}-x_{t_{j-1}}) - (t_j-t_{j-1})h(\theta_{t_j}) \\
\ge \: & \sup_{\theta \in C([0,T],J)} \!\limsup_\tau \sum_{j=1}^{\# \tau} \theta_{t_j} (x_{t_j}-x_{t_{j-1}}) - (t_j-t_{j-1})h(\theta_{t_j}) \\
= \: & \sup_{\theta \in C([0,T],J)} \int_0^T \theta_{t} dx_t - \int_0^T h(\theta_{t}) dt \:.
\end{align*}
To prove the other inequality, we use the following construction. Fix $\tau$ and let $\theta \in C([0,T],J)$. Let also $\epsilon > 0$ such that $\epsilon < \min (t_j-t_{j-1})$ and define $\theta^{\epsilon,\tau}$ as
\[
\theta^{\epsilon,\tau}_t = 
\begin{cases}
\theta_{t_{j-1}} + \frac{t-t_{j-1}}{\epsilon} \, (\theta_{t_{j}}-\theta_{t_{j-1}}) & \text{ if } t \in [t_{j-1},\,t_{j-1}+\epsilon] \:, \\
\theta_{t_{j}} & \text{ if } t \in [t_{j-1}+\epsilon,\, t_j] \:. \\
\end{cases}
\]
Then
\begin{align*}
& \left|\sum_{j=1}^{\# \tau} \theta_{t_j} (x_{t_j}-x_{t_{j-1}}) - (t_j-t_{j-1})h(\theta_{t_j}) - \int_0^T \theta_{t}^{\epsilon,\tau} dx_t + \int_0^T h(\theta_{t}^{\epsilon,\tau}) dt \right| \\
= \: & \left|\sum_{j=1}^{\# \tau} (\theta_{t_j} - \theta_{t_{j-1}}) \int_{t_{j-1}}^{t_{j-1}+\epsilon} \left(1-\frac{t-t_{j-1}}{\epsilon}\right) dx_t + \int_{t_{j-1}}^{t_{j-1}+\epsilon} h(\theta_{t}^{\epsilon,\tau})-h(\theta_{t_j}) dt \right| \\
\le \: & \sum_{j=1}^{\# \tau} \left|\theta_{t_j} - \theta_{t_{j-1}}\right| \left|\int_{t_{j-1}}^{t_{j-1}+\epsilon} \left(1-\frac{t-t_{j-1}}{\epsilon}\right) dx_t \right| \\
& \qquad\qquad\qquad\qquad\qquad\qquad\qquad\qquad+ 2 \epsilon \max\left\{ |h(\theta)|\,:\,\theta \in [\theta_{t_{j-1}},\theta_{t_{j}}]\right\} \\
\le \: & \sum_{j=1}^{\# \tau} \left|\theta_{t_j} - \theta_{t_{j-1}}\right| \, \mu_x\big(]0,\epsilon]\big) + 2 \epsilon \max\left\{ |h(\theta)|\,:\,\theta \in [\theta_{t_{j-1}},\theta_{t_{j}}]\right\} \underset{\epsilon \rightarrow 0}{\rightarrow} 0 \:,
\end{align*}
where $\mu_x$ is the measure associated with $x$. Hence 
\begin{align*}
\sup_{\theta \in C([0,T],J)} \!\sup_\tau \sum_{j=1}^{\# \tau} \theta_{t_j}& (x_{t_j}-x_{t_{j-1}}) - (t_j-t_{j-1})h(\theta_{t_j}) \\
& \le \sup_{\theta \in C([0,T],J)} \int_0^T \theta_{t} dx_t - \int_0^T h(\theta_{t}) dt 
\end{align*}
and 
\[
\Lambda^*(x) = \sup_{\theta \in C([0,T],J)} \int_0^T \theta_{t} dx_t - \int_0^T h(\theta_{t}) \, dt \:.
\]
We will now use \cite[Thm. 5.]{Roc1971} to obtain the result. Since $x$ has finite variation, the measure $dx_t$ is regular. Using the notation of \cite{Roc1971}, in our case the multifunction $D$ is the constant multifunction $t \mapsto D(t)=J$. Therefore $D$ is fully lower semi-continuous. Furthermore, since $[0,1] \subset J$, the interior of $D(t)$ is non-empty. The set $[0,T]$ is compact with no non-empty open sets of measure 0 and for every $u$ in the interior of $J$, and $V \in [0,T]$ open,
\[
\int_V |h(u)| \,dt \le T |h(u)| < \infty\:.
\]
\cite[Thm. 5.]{Roc1971} then implies that 
\[
\sup_{\theta \in C([0,T],J)} \int_0^T \theta_{t} \, dx_t - \int_0^T h(\theta_{t}) \, dt = \int_0^T h^*(\px{t}) \, dt + \int_0^T \H\left( \frac{d\nu_t}{d\theta_t} \right)\, d\theta_t \:,
\]
where 
\[
h^*(y) = \lim_{\epsilon \rightarrow 0} \sup_{\theta \in J} \left\{\theta y - h(\theta) \right\} \:, \qquad \H(y) = \lim_{\epsilon\rightarrow 0} \epsilon \, h^*(y/\epsilon) \:,
\]
$\px{\phantom{t}}$ is the derivative of the absolutely continuous part of $x$, 
$\nu_t$ is the singular component of $dx_t$ with respect to $dt$ and $\theta_t$ is any non-negative, finite, regular, $\R$-valued Borel measure, with respect to which $\nu_t$ is absolutely continuous. 
\end{proof}
\begin{remark}
In particular, the proof of Theorem \ref{thm:LPDInfiniteDimensionalRateFunction} shows that, if $x$ does not belong to $V_r$, the set of trajectories $x : [0,t] \rightarrow \R$ with bounded variation, then $\Lambda^*(x) = \infty$. 
\end{remark}

\section{Variance reduction} \label{sec:VarianceReduction}

Denote $P(S)$ the payoff of an option on $(S_t)_{0\le t \le T}$. The price of an option is generally calculated as the expectation $\E(P(S))$ under a certain risk-neutral measure $\P$. For any equivalent measure $\Q$, the price of the option can be written
\[
\E(P(S)) = \E^\Q\left(P(S) \frac{d\P}{d\Q}\right) \:.
\]
The variance of $P(S)$ is 
\[
\var_\P\left(P(S)\right) = \E\left(P^2(S)\right)-\E^2\left(P(S)\right) \:,
\]
whereas 
\begin{align*}
\var_\Q\left( P(S) \frac{d\P}{d\Q} \right) 
& = \E^\Q\left(P^2(S) \left(\frac{d\P}{d\Q}\right)^2\right)-\left(\E^\Q\left(P(S) \frac{d\P}{d\Q}\right)\right)^2 \\
& = \E\left(P^2(S) \frac{d\P}{d\Q}\right)- \E^2\left(P(S)\right) \:. 
\end{align*}
We can therefore choose $\Q$ in order to reduce the variance of the random variable, whose expectation gives the price of the option. \\
A flexible class of measure changes introduced in \cite{Ge:Ta} is given by path dependent Esscher transform, that is the class of measures $\P_\theta$ such that
\[
\frac{d\P_\theta}{d\P} = \frac{e^{\int_0^T X_t \,d\theta_t}}{\E\left[e^{\int_0^T X_t \,d\theta_t}\right]}\:,
\]
where $\theta$ belongs to $M$, the set of signed measures on $[0,T]$. Denoting $H(X)=\log P\left(S_0\,e^{X}\right)$, the optimization problem writes
\begin{equation}\label{eq:OptimizationProblemExact}
\inf_{\theta \in M} \E\left[ \exp\left( 2H(X) - \int_0^T X_t \,d\theta_t + \G_1(\theta) \right)\right]\:,
\end{equation}
where
\[
\G_\epsilon(\theta) := \epsilon\log \E\left[e^{\frac{1}{\epsilon}\int_0^T X_t^\epsilon \,d\theta_t}\right] \:.
\]
The optimization problem \eqref{eq:OptimizationProblemExact} cannot be solved explicitly. We therefore choose to solve the problem asymptotically using the two following lemmas. 
Denote $\bar{M}$ the set of measures $\theta \in M$ with support on a finite set of points. We first give a lemma that characterizes the behaviour of $\G_\epsilon(\theta)$ as $\epsilon\rightarrow 0$, for $\theta \in \bar{M}$ as this will be sufficient for the cases that we will consider in Section \ref{sec:NumericalExamples} (see Prop. \ref{prop:OptimalMeasureIsDiscrete}).
\begin{lemma}\label{lem:Gepsilon}
If Assumption \ref{ass:H2} holds, then for any measure $\theta\in \bar{M}$, such that for every $t \in [0,T]$, $\theta([t,T]) \in J$, we have
\[
\lim_{\epsilon \rightarrow 0}\G_\epsilon(\theta) = \int_0^T h(\theta([t,T])) \,dt \:.
\]
\end{lemma}
\begin{proof}
Denote $\tau = \{t_1,...,t_n\}$, the support of $\theta$. We then obtain
\begin{align*}
\lim_{\epsilon \rightarrow 0} \epsilon\log \E\left[e^{\frac{1}{\epsilon}\int_0^T X_t^\epsilon \,d\theta_t}\right] 
& = \lim_{\epsilon \rightarrow 0} \epsilon\log \E\left[ e^{\frac{1}{\epsilon}\sum_{j=1}^n X_{t_j}^\epsilon \,\theta\big((t_{j-1}, t_j\,]\big)}\right] \\
& = \sum_{j=1}^n (t_j-t_{j-1}) \, h\left(\theta\big((t_{j-1}, t_j\,]\big)\right) \\
& = \int_0^T h(\theta([t,T])) \,dt 
\end{align*}
by applying Theorem \ref{thm:LambdaTau} to $\theta=\left(\theta\big((t_0, t_1\,]\big),...,\theta\big((t_{n-1}, t_n\,]\big)\right)$.
\end{proof}
Next, we give a result that characterizes the behaviour of the variance minimization problem \eqref{eq:OptimizationProblemExact} where $X$ has been replaced by $X^\epsilon$ as $\epsilon \rightarrow 0$. 
\begin{lemma} \label{lem:LimitFunction}
Let $\theta \in \bar{M}$ such that $- \theta([t,T]) \in J^\circ$ for every $t \in [0,T]$. Assume that the assumptions of Theorem \ref{thm:LDPFiniteDimensional} hold. Assume furthermore that $H: \F([0,T],\R) \rightarrow \R$ is bounded from above by a constant $C$ and continuous on $D$ the set of functions $x \in V_r$, such that $H(x) > -\infty$, with respect with to the pointwise convergence topology. Then 
\begin{align*}
& \lim_{\epsilon \rightarrow 0} \epsilon \log \E\left[ \exp\left(\frac{2H(X^\epsilon) - \int_0^T X_t^\epsilon \,d\theta_t +\G_\epsilon(\theta)}{\epsilon} \right)\right] \\
& \hspace*{3cm} = \sup_{x\in D} \left\{ 2H(x) - \int_0^T x_t d\theta_t - \Lambda^*(x) \right\} + \int_0^T h(\theta([t,T])) \,dt \:.
\end{align*} 
\end{lemma}
\begin{proof}
First note that, by Lemma \ref{lem:Gepsilon}, 
\begin{align*}
& \lim_{\epsilon \rightarrow 0} \epsilon \log \E\left[ \exp\left(\frac{2H(X^\epsilon) - \int_0^T X_t^\epsilon \,d\theta_t +\G_\epsilon(\theta)}{\epsilon} \right)\right] \\
& \hspace*{2cm} = \lim_{\epsilon \rightarrow 0} \epsilon \log \E\left[ \exp\left(\frac{2H(X^\epsilon) - \int_0^T \! X_t^\epsilon \,d\theta_t }{\epsilon} \right)\right] + \int_0^T h(\theta([t,T])) \,dt \:.
\end{align*}
We therefore just need to prove that 
\[
\lim_{\epsilon \rightarrow 0} \epsilon \log \E\!\left[ \exp\left(\frac{2H(X^\epsilon) - \int_0^T \!\! X_t^\epsilon \,d\theta_t }{\epsilon} \right)\right]\! = \sup_{x\in D} \left\{ 2H(x) - \!\int_0^T \!\!\! x_t d\theta_t - \Lambda^*(x) \right\} \:.
\]
Denote $\varphi: \,\F([0,T],\R) \rightarrow \R$ the function $\varphi(x) = 2H(x) - \int_0^T x_t \,d\theta_t$. Since $H$ is assumed to be continuous and $\theta$ has support on $\tau$, $\varphi$ is continuous. Let us show the integrability condition of Theorem \ref{thm:Varadhan}. For every $\gamma$
\begin{align*}
& \hspace*{-1.5cm}\limsup_{\epsilon \rightarrow 0} \epsilon\,\log \E\left[\exp\left(\frac{\gamma\,\varphi(X^\epsilon)}{\epsilon} \right) \right] \\
& = \limsup_{\epsilon \rightarrow 0} \epsilon \log \E\left[ \exp\left(\frac{2\gamma H(X^\epsilon) - \gamma \int_0^T X_t^\epsilon \,d\theta_t}{\epsilon} \right)\right] \\
& \le 2 \gamma C + \limsup_{\epsilon \rightarrow 0} \epsilon\log \E\left[e^{\frac{1}{\epsilon}\int_0^T X_t^\epsilon \,d(-\gamma\theta)_t}\right] \:.
\end{align*}
Since $- \theta([t,T]) \in J^\circ$ for every $t \in [0,T]$, there exists $\gamma > 1$ such that $-\gamma\theta([t,T])$ remains in $J$ for every $t$. Therefore Lemma \ref{lem:Gepsilon} applies and  
\[
\limsup_{\epsilon \rightarrow 0} \epsilon\,\log \E\left[\exp\left(\frac{\gamma\,\varphi(X^\epsilon)}{\epsilon} \right) \right] \le 2 \gamma C + \int_0^T h(-\gamma \, \theta([t,T])) \,dt < \infty \:.
\]
Theorem \ref{thm:Varadhan} then applies and yields the result.
\end{proof}
\begin{definition}
Let $\theta \in M$. We say that $\theta$ is \textit{asymptotically optimal} if it minimises
\[
\sup_{x\in V_r} \left\{ 2H(x) - \int_0^T x_t \, d\theta_t - \Lambda^*(x) \right\} + \int_0^T h(\theta([t,T])) \,dt \:.
\]
\end{definition}
In general, $\Lambda^*$ is not easy to calculate explicitly. To solve this problem, we cite the following theorem of \cite{Ge:Ta}.
\begin{theorem}\label{thm:Equivalence}
Let $H$ be concave and assume that the set $\{ x \in V_r \,:\, H(x) > -\infty \}$ is non-empty and contains a constant element. Assume furthermore that $H$ is continuous on this set with respect to the topology of pointwise convergence, that  $h$ is lower semi-continuous with open and bounded effective domain and that there exists a $\lambda > 0$ such that $h$ is complex-analytic on $\{z \in \C \,:\, |\text{Im}(z)| < \lambda\}$. 
Then
\begin{align*}
\inf_{\theta \in M} \sup_{x\in V_r} \left\{ 2H(x) - \int_0^T x_t d\theta_t - \Lambda^*(x) \right\} & + \int_0^T h(\theta([t,T])) \,dt \\ 
& \hspace*{-1cm} = 2 \, \inf_{\theta \in M} \left\{ \hat{H}(\theta) + \int_0^T h(\theta([t,T])) \,dt \right\} \:,
\end{align*}
where 
\[
\hat{H}(\theta) = \sup_{x\in V_r} \left\{ H(x) - \int_0^T x_t \, d\theta_t \right\} \:.
\]
Furthermore, if $\theta^*$ minimises the left-hand side of the above equation, it also minimises the right-hand side.
\end{theorem}
We finally give a result for the case where $H$ depends on $x$ only through $x_{t_1},....,x_{t_n}$.
\begin{proposition}\label{prop:OptimalMeasureIsDiscrete}
Let $\tau = \{t_1,...,t_n\}$ and let $H: \F([0,t],\R) \rightarrow \R \cup \{-\infty\}$ be a log-payoff depending on $x$ only through $x_\tau$. Then for every $\theta \in M$ such that $\theta(\tau) \neq \theta([0,T])$, $\hat{H}(\theta) = \infty$.
\end{proposition}
\begin{proof}
Assume that $\theta \in M$ is such that $\theta(\tau) \neq \theta([0,T])$. Then there exists a set $A \subset [0,T] \backslash \tau$, such that $\theta(A) \neq 0$. Fix $\bar{x} \in D$. By definition, $H(\bar{x}) > -\infty$. Then
\[
H(\hat{x} + \alpha \,\mathds{1}_A) - \int_0^T (\bar{x}_t + \alpha \,\mathds{1}_A) d\theta_t = H(\hat{x}) - \int_0^T \bar{x}_t d\theta_t - \alpha \,\theta(A) \:.
\]
By letting $\alpha$ tend to $\text{sgn}(\theta) \, \infty$, one can therefore increase indefinitely $H(x) - \int_0^T x_t d\theta_t$. Therefore, $\hat{H}(\theta) = \infty$.
\end{proof}

\section{Numerical examples} \label{sec:NumericalExamples}

In this section, we apply the variance reduction method to several examples. We first prove a result for options on the average value of the underlying over a finite set of points.  
\begin{proposition}\label{prop:HHat}
Let $\tau=\{t_1,...,t_n\}$ and consider an option with log-payoff 
\[
H(x) = \log\left( K - \frac{S_0}{n} \sum_{j=1}^n e^{x_{t_j}} \right)_+ \:.
\]
Then for any $\theta \in \bar{M}$ with support on $\theta = \{t_1,...,t_n\}$,
\begin{equation} \label{eq:HHat}
\hat{H}(\theta) = \log\left( \frac{K}{1 - \sum_{l=1}^n \theta_l} \right) - \sum_{m=1}^n \theta_m \log\left(\frac{-\theta_m\,n\,K/S_0}{1 - \sum_{l=1}^n \theta_l}\right)
\end{equation}
where we use the abuse of notation $\theta_j = \theta(\{t_j\})$. 
\end{proposition}
\begin{proof}
In this case,
\[
H(x) - \int_0^T x_t d\theta_t = \log\left( K - \frac{S_0}{n} \sum_{j=1}^n e^{x_{t_j}} \right)_+ \!\!\! -  \, \sum_{j=1}^n \theta_j x_{t_j} \:.
\]
When the option is out or at the money, the log-payoff is $-\infty$. Assume that $x$ is such that $H(x) > -\infty$ and differentiate with respect to $x_{t_j}$. We obtain
\[
0 = \de{x_j}\left\{\log\left( K - \frac{S_0}{n} \sum_{l=1}^n e^{x_l} \right) - \sum_{l=1}^n x_l \theta_l \right\} = \frac{- \frac{S_0}{n} \, e^{x_j}}{ K - \frac{S_0}{n} \sum_{l=1}^n e^{x_l}} - \theta_j \:.
\]
Therefore the $x$ that maximises $H(x) - \int_0^t x_s d\theta_s$ satisfies
\[
\frac{e^{x_{t_j}}}{\theta_j} =  - n\,\frac{K}{S_0} + \sum_{l=1}^n e^{x_{t_l}} =  - n\,\frac{K}{S_0} + \frac{e^{x_{t_j}}}{\theta_j} \sum_{l=1}^n \theta_l \:,
\]
for every $j$. Therefore
\[
x_{t_j} = \log\left(\frac{-\theta_j\,n\,K/S_0}{1 - \sum_{l=1}^n \theta_l}\right) \:.
\]
Inserting $x_{t_j}$ in the value of $H(x) - \int_0^T x_t \, d\theta_t$, we obtain the result.
\end{proof}

\subsection{European and Asian put options in the Heston model}

Consider the Heston model \cite{Hes1993}
\begin{equation}\label{eq:HestonDynamicsP}
\begin{aligned}
dX_t & = -\frac{V_t}{2} \, dt + \sqrt{V_t} \, dW_t^1 \:,  && X_0 = 0\\
dV_t & = \lambda( \mu - V_t) \, dt + \zeta \sqrt{V_t} \, dW_t^2 \:,  && V_0 > 0 \\
& \hspace*{-.5cm} d\q{W^1,W^2}_t = \rho \, dt \:,
\end{aligned}
\end{equation}
where $W^1,W^2$ are standard $\P$-Brownian motions. The Laplace transform of $(X_t,V_t)$ is
\[
\E\left( e^{u X_t + w V_t} \right) = e^{ \phi(t,u,w) + \psi(t,u,w) V_0 + u X_0 } \:,
\]
where $\phi, \psi$ satisfy the Riccati equations 
\begin{equation}\label{eq:RiccatiEquations}
\begin{aligned}
\de{t} \phi(t,u,w) & = F(u, \psi(t,u,w)) \quad && \phi(0,u,w) = 0 \\
\de{t} \psi(t,u,w) & = R(u, \psi(t,u,w)) \quad && \psi(0,u,w) = w
\end{aligned}
\end{equation}
for $F(u,w) = \lambda \mu w$ and 
\[
R(u,w) = \frac{\zeta^2}{2} \, w^2 + \zeta \rho \, u w - \lambda w + \frac{1}{2} (u^2 - u) \:.
\]
A standard calculation shows that the solution of the Riccati equations \eqref{eq:RiccatiEquations} is
\begin{equation}\label{eq:RiccatiSolution}
\begin{aligned}
\psi(t,u,w) & = \frac{1}{\zeta} \left(\frac{\lambda}{\zeta} - \rho u \right) - \frac{\gamma}{\zeta^2} \,\frac{\tanh\left(\frac{\gamma}{2} \, t\right) + \eta}{1 + \eta \, \tanh\left(\frac{\gamma}{2} \, t\right)}  \\
\phi(t,u,w) & = \mu\,\frac{\lambda}{\zeta}\left(\frac{\lambda}{\zeta} - \rho u \right) t - 2\mu \,\frac{\lambda}{\zeta^2}\,\log\left(\cosh\left(\frac{\gamma}{2} \, t\right) + \eta \,\sinh\left(\frac{\gamma}{2} \, t\right)\right) \:,
\end{aligned}
\end{equation}
where $\gamma = \gamma(u) = \zeta \, \sqrt{\left(\frac{\lambda}{\zeta} - \rho u\right)^2 \!\!+ \frac{1}{4} - \!\left(u-\frac{1}{2}\right)^2}$ and $\eta = \eta(u,w) = \frac{\lambda - \zeta\rho u - \zeta^2 w}{\gamma(u)}$. 
Furthermore, for the Heston model, the function $h$ is given by
\begin{equation}\label{eq:HestonH}
h(u) = \mu\,\frac{\lambda}{\zeta}\left(\frac{\lambda}{\zeta} - \rho u \right) - \mu\,\frac{\lambda}{\zeta^2} \, \gamma(u) \:.
\end{equation}
\begin{remark}
The log-Laplace transform of the Heston model converges to the log-Laplace transform $h$  of an NIG process \cite{Bar1997}, which is complex-analytic on a strip around the real axis, thus allowing to apply Theorem \ref{thm:Equivalence}.
\end{remark}
The following proposition describes the effect of the time dependent Esscher transform on the dynamics of the Heston model.
\begin{proposition}\label{prop:HestonDynamicsPTheta}
Let $\tau = \{t_1,...,t_n\}$ and $\P_\theta$ the measure given by 
\[
\frac{d\P_\theta}{d\P} = \frac{e^{\sum_{j=1}^n \theta_j \, X_{t_j}}}{\E\left[e^{\sum_{j=1}^n \theta_j \, X_{t_j}}\right]} \:.
\] 
Under $\P_\theta$, the dynamics of the $\P$-Heston process $(X_t,V_t)$ becomes
\begin{equation}\label{eq:HestonDynamicsPTheta}
\begin{aligned}
dX_t & = \left(\Theta_{\tau_t} + \zeta \rho \,\Psi\left(\tau_t-t, \Theta_{\tau_t}, ..., \Theta_n \right) -\frac{1}{2} \right) V_t \, dt + \sqrt{V_t} \, d\tilde{W}_t^1 \:,  && X_0 = 0\\
dV_t & = \tilde{\lambda}_t\, ( \tilde{\mu}_t - V_t) \, dt + \zeta \sqrt{V_t} \, d\tilde{W}_t^2 \:,  && V_0 = V_0 \\
& \hspace*{-.5cm} d\q{\tilde{W}^1,\tilde{W}^2}_t = \rho \, dt \:,
\end{aligned}
\end{equation}
where $\tilde{W}$ is 2-dimensional correlated $\P_\theta$-Brownian motion, $\Theta_j = \sum_{m=j}^n \theta_m$, and $\Phi$ and $\Psi$ are defined iteratively as
\begin{align*}
\Psi\left(s, \Theta_{j}, ..., \Theta_n \right) 
& = \psi\left(s, \Theta_{j}, \Psi\left(t_{j+1}-t_j, \Theta_{j+1} ..., \Theta_n \right)\right) \\
\Psi\left(s \right) 
& = 0 \\
\Phi\left(s, \Theta_{j}, ..., \Theta_n \right) 
& = \phi\left(s, \Theta_{j}, \Psi\left(t_{j+1}-t_j, \Theta_{j+1}, ..., \Theta_n \right)\right) \\
& \qquad+ \Phi\left(t_{j+1}-t_j, \Theta_{j+1}, ..., \Theta_n \right) \\
\Phi\left(s \right) 
& = 0
\end{align*}
and where, denoting $\tau_t = \inf\{s \in \tau \,:\, s \ge t \}$,
\[
\tilde{\lambda}_t = \lambda - \zeta \Theta_{\tau_t} \rho - \zeta^2 \,\Psi\left(\tau_t-t, \Theta_{\tau_t}, ..., \Theta_n \right) \quad \text{and} \quad \tilde{\mu}_t = \frac{\lambda \mu}{\tilde{\lambda}_t} \:.
\]
\end{proposition}
\begin{proof}
Denote
\[
D(t,X_t,V_t) = \left.\frac{d\P_\theta}{d\P}\right|_{\F_t} \:.
\]
Then 
\begin{align*}
D(t,X_t,V_t) 
& = \frac{e^{\sum_{j=1}^{\tau_t-1} \theta_j \, X_{t_j}}}{\E\left[e^{\sum_{j=1}^n \theta_j \, X_{t_j}}\right]} \: \E\left[e^{\sum_{j=\tau_t}^{n} \theta_j \, X_{t_j}} \, \middle|\,\F_t\right] \\
& = \frac{e^{\sum_{j=1}^{\tau_t-1} \theta_j \, X_{t_j} + \Phi\left(\tau_t-t, \Theta_{\tau_t}, ..., \Theta_n \right)}}{e^{\Phi\left(t_1, \Theta_{1}, ..., \Theta_n \right) + \Psi\left(t_1, \Theta_{1}, ..., \Theta_n \right) \, V_0 + \Theta_{1} \, X_{0} }} \: e^{ \Psi\left(\tau_t-t, \Theta_{\tau_t}, ..., \Theta_n \right) \, V_t + \Theta_{\tau_t} \, X_{t} } \:.
\end{align*}
The dynamics of $D(t,X_t,V_t)$ can then be expressed using It\=o's Lemma as
\begin{align*}
dD(t,X_t,V_t) 
& = D(t,X_t,V_t) \left(\Theta_{\tau_t} dX_t + \Psi\left(\tau_t-t, \Theta_{\tau_t}, ..., \Theta_n \right) dV_t\right) + ...\, dt\\
& = D(t,X_t,V_t) \sqrt{V_t}\left(\Theta_{\tau_t} dW_t^1 + \zeta \, \Psi\left(\tau_t-t, \Theta_{\tau_t}, ..., \Theta_n \right) dW_t^2\right) \:.
\end{align*}
By Girsanov's theorem, 
\[
d\left(\begin{array}{c}
\tilde{W}_t^1 \\
\tilde{W}_t^2
\end{array}\right)
=
d\left(\begin{array}{c}
W_t^1 \\
W_t^2
\end{array}\right) 
- \sqrt{V_t}
\left(\begin{array}{c}
\Theta_{\tau_t} + \zeta \rho \,\Psi\left(\tau_t-t, \Theta_{\tau_t}, ..., \Theta_n \right) \\
\Theta_{\tau_t} \rho + \zeta \,\Psi\left(\tau_t-t, \Theta_{\tau_t}, ..., \Theta_n \right)
\end{array}\right) dt 
\]
is a 2-dimensional Brownian motion under the measure $\P_\theta$. Replacing $W$ in eq. \eqref{eq:HestonDynamicsP} by $\tilde{W}$ gives the result.
\end{proof}
\begin{remark}
Prop. \ref{prop:HestonDynamicsPTheta} shows that the time-dependent Esscher transform changes a classical Heston process into a Heston process with time-inhomogeneous drift.
\end{remark}
\begin{remark}
Note that Assumption \ref{ass:H2} is verified in the Heston model only when $\rho = 0$. Indeed, $J=[u_-,u_+]$, where 
\[
u_\pm = \frac{\left(\frac{1}{2} - \frac{\lambda}{\zeta}\,\rho \right) \pm \sqrt{\left(\frac{1}{2} - \frac{\lambda}{\zeta}\,\rho \right)^2 + \frac{\lambda^2}{\zeta^2} \, (1 - \rho^2)}}{(1 - \rho^2)} \:,
\]
while
\[
w(u_-) = \frac{1}{\zeta}\left(\frac{\lambda}{\zeta} - \rho u_- \right) \qquad \text{and} \qquad w(u_+) = \frac{1}{\zeta}\left(\frac{\lambda}{\zeta} - \rho u_+ \right) \:.
\]
However, since the actual variance reduction problem is itself unsolvable, our goal is to find a good candidate measure that we can test numerically. The fact that we do not have the full theory to justify it is therefore not problematic. 
\end{remark}

\subsubsection{Numerical results for European put options}

In this case, by Prop. \ref{prop:OptimalMeasureIsDiscrete} with $n=1$ and $t_1=T$, $\theta$ has support on $\{T\}$. Using the abuse of notation $\theta := \theta(\{T\})$, we have 
\begin{equation}\label{eq:HHatHestonEuro}
\begin{aligned}
& \quad \hat{H}(\theta) + \int_0^T h(\theta([t,T])) \,dt \\
& = \: \log\left( \frac{K}{1 - \theta} \right) - \theta \log\left(\frac{-\theta\,K/S_0}{1 - \theta}\right) + T \,\mu\,\frac{\lambda}{\zeta} \left(\frac{\lambda}{\zeta} - \rho \, \theta  - \frac{\gamma(\theta)}{\zeta} \, \right) \:.
\end{aligned}
\end{equation}
In order to obtain $\theta$, we therefore differentiate \eqref{eq:HHatHestonEuro} with respect to $\theta$ and equate the derivative to 0 by dichotomy . \\
We simulate $N=10000$ trajectories of the Heston model with parameters $\lambda = 1.15$, $\mu = 0.04$, $\zeta = 0.2$, $\rho = -0.4$ and initial values $V_0 = 0.04$ and $S_0 = 1$, under both $\P$, eq. \eqref{eq:HestonDynamicsP}, and $\P_\theta$, eq. \eqref{eq:HestonDynamicsPTheta}, with $n=1$ and $t_1 = T$, using a standard Euler scheme with 200 discretization steps. For the $\P$-realisations $X^{(i)}$, we calculate the European put price as $\frac{1}{N} \sum_{j=1}^N \left(K-S_0 \, e^{X_T^{(i)}}\right)_+$
and for the $\P_\theta$-realisations $X^{(i,\theta)}$, as
\begin{equation}\label{eq:EstimatorTheta}
\frac{e^{\phi\left(T, \theta, 0 \right) + \psi\left(T, \theta, 0 \right) \, V_0 }}{N} \sum_{j=1}^N e^{-\theta \, X_{T}^{(i,\theta)}} \left(K-S_0 e^{X_T^{(i,\theta)}}\right)_+ \:.
\end{equation}
Each time, we compute the $\P_\theta$-standard deviation, the variance ratio and the adjusted variance ratio, i.e. the variance ratio divided by the ratio of simulation time. The latter measures the actual efficiency of the method, given the fact that simulating under the measure change takes in general slightly more time. \\
In Table \ref{tab:PutVarianceReductionMaturity}, we fix the strike to the value $K=1$ and let the maturity $T$ vary from $0.25$ to $3$, whereas in Tables \ref{tab:PutVarianceReductionStrikeT1} and \ref{tab:PutVarianceReductionStrikeT3}, we fix maturity to $T=1$ and to $T=3$, while we let the strike $K$ vary between $0.25$ and $1.75$. We calculate each time the price, the standard error, the variance ratio adjusted and not adjusted by the ratio of simulation time.
\begin{table}[H]
\centering
\begin{tabular}{cccccc}
\hline
$T$ & Price & Std. error & Var. ratio & Adj. ratio & Time, s \\
\hline
\hline
0.25 & 0.0395 & 3.72 $\cdot 10^{-4}$ & 2.46 & 2.14 & 20.2 \\
0.5 & 0.0550 & 4.54 $\cdot 10^{-4}$ & 3.12 & 2.83 & 19.9 \\
1 & 0.0780 & 5.59 $\cdot 10^{-4}$ & 3.92 & 3.66 & 19.5 \\
2 & 0.111 & 7.20 $\cdot 10^{-4}$ & 4.21 & 3.89 & 19.7 \\
3 & 0.134 & 8.48 $\cdot 10^{-4}$ & 4.19 & 3.79 & 19.8 \\
\hline
\end{tabular}
\caption{The variance ratio as function of the maturity for at-the-money European put options.}
\label{tab:PutVarianceReductionMaturity}
\end{table}
\begin{table}[H]
\centering
\begin{tabular}{cccccc}
\hline
$K$ & Price & Std. error & Var. ratio & Adj. ratio & Time, s \\
\hline
\hline
0.5 & 0.00014 & 7.65 $\cdot 10^{-6}$ & 26.6 & 24.5 & 18.4 \\
0.75 & 0.00794 & 1.34 $\cdot 10^{-4}$ & 6.53 & 5.91 & 18.7 \\
1 & 0.0773 & 5.60 $\cdot 10^{-4}$ & 3.96 & 3.65 & 18.5 \\
1.25 & 0.261 & 8.62 $\cdot 10^{-4}$ & 4.20 & 3.78 & 18.9 \\
1.5 & 0.502 & 7.92 $\cdot 10^{-4}$ & 5.84 & 5.36 & 18.6 \\
1.75 & 0.749 & 6.84 $\cdot 10^{-4}$ & 8.45 & 7.29 & 19.7 \\
\hline
\end{tabular}
\caption{The variance ratio as function of the strike for the European put option with maturity $T=1$.}
\label{tab:PutVarianceReductionStrikeT1}
\end{table}
\begin{table}[H]
\centering
\begin{tabular}{cccccc}
\hline
$K$ & Price & Std. error & Var. ratio & Adj. ratio & Time, s \\
\hline
\hline
0.25 & 7.1 $\cdot 10^{-5}$ & 1.84 $\cdot 10^{-5}$ & 92.0 & 70.9 & 23.1 \\
0.5 & 0.00418 & 6.05 $\cdot 10^{-5}$ & 16.1 & 16.0 & 20.0 \\
0.75 & 0.0369 & 3.43 $\cdot 10^{-4}$ & 6.67 & 6.00 & 20.4 \\
1 & 0.133 & 8.51  $\cdot 10^{-4}$ & 4.24 & 4.15 & 20.2 \\
1.25 & 0.300 & 1.34 $\cdot 10^{-3}$ & 3.61 & 3.13 & 21.3 \\
1.5 & 0.517 & 1.60 $\cdot 10^{-3}$ & 3.47 & 3.30 & 19.9 \\
1.75 & 0.755 & 1.64 $\cdot 10^{-3}$ & 3.89 & 3.53 & 19.9 \\
\hline
\end{tabular}
\caption{The variance ratio as function of the strike for the European put option with maturity $T=3$.}
\label{tab:PutVarianceReductionStrikeT3}
\end{table}
In all the cases, we can see that the variance ratio becomes very interesting when the option gets deeply out of the money and less significant, yet still very  interesting, when the option is at or in the money. This corresponds to the natural behaviour of variance reduction techniques that involve measure changes, as the measure change is going to increase the probability of choosing a trajectory that is eventually going to enter the money. Note that the simulation time is only slightly larger when simulating with the measure change, while the time required for the optimization procedure is negligible compared with the simulation time. In Figure \ref{fig:VarianceHeston}, we fix the maturity to $T=1.5$ and plot the empirical variance of the estimator \eqref{eq:EstimatorTheta} as a function of $\theta$. Our method provides $\theta = -0.457$ as asymptotically optimal measure change. We can therefore see that the asymptotically optimal $\theta$ is very close to the optimal one.
\begin{figure}[H]
\centering
\includegraphics[width=10cm,keepaspectratio=true]{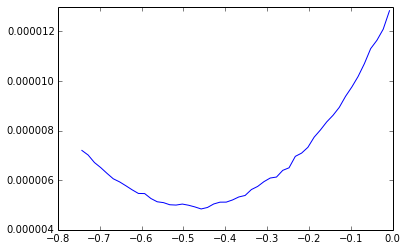}
\caption{The variance of the Monte-Carlo estimator as a function of $\theta$.}
\label{fig:VarianceHeston}
\end{figure}

\subsubsection{Numerical results for Asian put options}

We now consider the case of a (discretized) Asian put option. Here, the log-payoff is
\[
H(X) = \log\left(K-\frac{S_0}{n}\sum_{j=1}^n e^{X_{t_j}}\right)_+ \:,
\]
where $t_j = \frac{j}{n}\,T$. By Prop. \ref{prop:OptimalMeasureIsDiscrete}, the support of $\theta$ is $\{t_1,...,t_n\}$ and we can denote $\theta_j = \theta(\{t_j\})$. Using Prop. \ref{prop:HHat} and eq. \eqref{eq:HestonH}, the function that we need to minimize is
\[
\log\left( \frac{K}{1 - \sum_{l=1}^n \theta_l} \right) - \sum_{m=1}^n \theta_m \log\left(\frac{-\theta_m\,n\,K/S_0}{1 - \sum_{l=1}^n \theta_l}\right) + \frac{T}{n} \sum_{j=1}^n h\left(\sum_{l=j}^n \theta_l\right) 
\]
or, alternatively, denoting $\Theta_{j} = \sum_{l=j}^n \theta_l$,
\[
\log\left( \frac{K}{1 - \Theta_1} \right) - \sum_{m=1}^n (\Theta_m - \Theta_{m+1})\log\!\left(\frac{-(\Theta_m \!- \Theta_{m+1})\,nK/S_0}{1 - \Theta_1}\right) + \frac{T}{n} \sum_{j=1}^n h\left(\Theta_j\right) \:. 
\]
By differentiating with respect to $\Theta_j$, we obtain, for $j=2,...,n$,
\begin{equation}\label{eq:MinimizationEquation2toN}
\begin{aligned}
0 
& = \de{\Theta_j}\left\{ \hat{H}(\theta) + \frac{T}{n}\sum_{m=1}^n h\left(\Theta_m \right) \right\} \\
& = \frac{T\,h'\left(\Theta_j \right)}{n} - \log\left[-(\Theta_j - \Theta_{j+1})\right] + \log\left[-(\Theta_{j-1} - \Theta_{j})\right] \:,
\end{aligned}
\end{equation}
while, for $j=1$, we have
\begin{equation}\label{eq:MinimizationEquation1}
\begin{aligned}
0 
& = \de{\Theta_1}\left\{ \hat{H}(\theta) + \frac{T}{n}\sum_{m=1}^n h\left(\Theta_m \right) \right\} \\
= & \log\left( 1 - \Theta_1 \right) - \log(n\,K/S_0) + \frac{T}{n}\,h'\left(\Theta_1 \right) - \log\left[-(\Theta_1 - \Theta_{2})\right] \:.
\end{aligned}
\end{equation}
Finally, taking the exponential in eqs. \eqref{eq:MinimizationEquation2toN} and \eqref{eq:MinimizationEquation1}, we obtain
\begin{align*}
\Theta_2-\Theta_{1\phantom{-1}} & = \: \: (1\phantom{\Theta_{n}\:}-\Theta_{1\phantom{-1}}) \: \: e^{\frac{T}{n}\,h'(\Theta_1)} \cdot \frac{S_0}{n\,K} \\
\Theta_3-\Theta_{2\phantom{-1}} & = \: \: (\Theta_{2\phantom{-1}}-\Theta_{1\phantom{-1}}) \: \: e^{\frac{T}{n}\,h'(\Theta_2)} \\
\vdots \:\:\qquad & = \qquad\qquad \vdots \\
\Theta_n-\Theta_{n-1} & = \: \: (\Theta_{n-1}-\Theta_{n-2}) \: \: e^{\frac{T}{n} \, h'(\Theta_{n-1})} \\
-\Theta_{n\phantom{-1}} & = \: \: (\Theta_{n\phantom{-1}}-\Theta_{n-1}) \: \: e^{\frac{T}{n} \, h'(\Theta_{n})} \:.
\end{align*}
Finally, define $\mathcal{T}$ the real function that associates to $\Theta_n$ 
\[
\mathcal{T}(\Theta_n) = (1-\Theta_{1}) \e^{\frac{T}{n}\,h'(\Theta_1)} \cdot \frac{S_0}{n\,K} - \Theta_2-\Theta_{1} \:,
\]
where $\Theta_{n-1} = \Theta_{n} + \Theta_{n} \, e^{-\frac{T}{n} \, h'(\Theta_{n})}$ and iteratively,
\[
\Theta_{j-2} = \Theta_{j-1} - (\Theta_j-\Theta_{j-1}) \, e^{-\frac{T}{n} \, h'(\Theta_{j-1})} \:, \qquad j=n,...,3 \:.
\]
Equating $\mathcal{T}$ to 0 by dichotomy then gives the asymptotically optimal measure. \\
Again, we simulate $N=10000$ trajectories of the Heston model with parameters $\lambda = 1.15$, $\mu = 0.04$, $\zeta = 0.2$, $\rho = -0.4$ and initial values $V_0 = 0.04$ and $S_0 = 1$, under both $\P$, eq. \eqref{eq:HestonDynamicsP}, and $\P_\theta$, eq. \eqref{eq:HestonDynamicsPTheta}, with $n=200$ and $t_j=\frac{j}{n}\,T$, using a standard Euler scheme with 200 discretization steps. For the $\P$-realisations $X^{(i)}$, we calculate the Asian put price as 
\begin{equation}
\frac{1}{N} \sum_{j=1}^N \left(K-\frac{S_0}{n}\sum_{j=1}^n e^{X_{t_j}^{(i)}}\right)_+
\end{equation}
and for the $\P_\theta$-realisations $X^{(i,\theta)}$, as
\begin{equation}
\frac{e^{\Phi\left(t_1, \Theta_{1}, ..., \Theta_n \right) + \Psi\left(t_1, \Theta_{1}, ..., \Theta_n \right) \, V_0}}{N} \sum_{j=1}^N e^{-\sum_{j=1}^n \theta_j \, X_{t_j}^{(i,\theta)}} \left(K-\frac{S_0}{n}\sum_{j=1}^n e^{X_{t_j}^{(i)}}\right)_+ \:.
\end{equation}
Again, each time, we compute the $\P_\theta$-standard deviation and the adjusted and non-adjusted variance ratios. 
In Table \ref{tab:AsianPutVarianceReductionMaturity}, we fix maturity to $T = 1.5$ and let the strike $K$ vary between $0.6$ and $1.3$.
\begin{table}[H]
\centering
\begin{tabular}{cccccc}
\hline
$K$ & Price & Std. error & Var. ratio & Adj. ratio & Time, s \\
\hline
\hline
0.6 & 3.466 $\cdot 10^{-5}$ & 4.13 $\cdot 10^{-6}$ & 16.9 & 14.6 & 19.9 \\
0.7 & 0.000562 & 2.60 $\cdot 10^{-5}$ & 5.77 & 4.77 & 21.1 \\
0.8 & 0.00414 & 9.64 $\cdot 10^{-5}$ & 4.36 & 3.77 & 20.1 \\
0.9 & 0.0185 & 0.00024 & 3.48 & 3.09 & 20.6 \\
1 & 0.0558 & 0.00043 & 3.49 & 3.07 & 20.1 \\
1.1 & 0.120 & 0.00057 & 3.69 & 3.20 & 20.1 \\
1.2 & 0.206 & 0.00062 & 4.27 & 3.80 & 19.7 \\
1.3 & 0.301 & 0.00059 & 5.30 & 4.41 & 21.0 \\
\hline
\end{tabular}
\caption{The variance ratio as function of the strike for the Asian put option. $\lambda = 1.15$, $\mu = 0.04$, $\zeta = 0.2$, $\rho = -0.4$, $S_0 = 1$, $V_0 = 0.04$, $T = 1.5$, $N = 10000$, $200$ discretization steps.}
\label{tab:AsianPutVarianceReductionMaturity}
\end{table}
The conclusion is the same as for the European put option. Indeed, the variance ratio explodes when the option moves away from the money. Due to the time-dependence of the measure change, the adjusted variance ratio is consistently around 13\% below its non-adjusted version. The adjusted variance ratio remains however very interesting, with values above 3 around the money.

\subsection{European put  options in the Heston model with negative exponential jumps}

We now consider the Heston model with negative exponential jumps
\begin{equation}\label{eq:HestonWithJumpsDynamicsP}
\begin{aligned}
dX_t & = \left( \delta-\frac{V_t}{2}\right) \, dt + \sqrt{V_t} \, dW_t^1 + dJ_t \:,  && X_0 = 0\\
dV_t & = \lambda( \mu - V_t) \, dt + \zeta \sqrt{V_t} \, dW_t^2 \:,  && V_0 = V_0 \\
& \hspace*{-.5cm} d\q{W^1,W^2}_t = \rho \, dt \:,
\end{aligned}
\end{equation}
where $W^1,W^2$ are standard $\P$-Brownian motions and $(J_t)_{t \ge 0}$ is an independent compound Poisson process with constant jump rate $r$ and jump distribution $\text{Neg-}\text{Exp}(\alpha)$, i.e. the L\'evy measure of $(J_t)_{t \ge 0}$ is $\nu(dx)=r \, \alpha e^{\alpha x} \mathds{1}_{\{x<0\}} dx$. The martingale condition on $S = S_0 \, e^X$ imposes $\delta=\frac{r}{\alpha+1}$.
The Laplace transform of $(X_t,V_t)$ is
\[
\E\left( e^{u X_t + w V_t} \right) = e^{ \phi(t,u,w) + \psi(t,u,w) V_0 + u X_0 } \:,
\]
where $\phi, \psi$ satisfy the Riccati equations 
\begin{equation}\label{eq:JumpRiccatiEquations}
\begin{aligned}
\de{t} \phi(t,u,w) & = F(u, \psi(t,u,w)) \quad && \phi(0,u,w) = 0 \\
\de{t} \psi(t,u,w) & = R(u, \psi(t,u,w)) \quad && \psi(0,u,w) = w
\end{aligned}
\end{equation}
for $F(u,w) = \lambda \mu \, w + \tilde{\kappa}(u)$, where $\tilde{\kappa}(u) = \frac{r u(u-1)}{(\alpha+1)(\alpha+u)}$, and 
\[
R(u,w) = \frac{\zeta^2}{2} \, w^2 + \zeta \rho \, u w - \lambda w + \frac{1}{2} (u^2 - u) \:.
\]
Again, a standard calculation shows that the solution of the generalized Riccati equations \eqref{eq:JumpRiccatiEquations} is
\begin{equation}\label{eq:GRiccatiSolution}
\begin{aligned}
\psi(t,u,w) & = \frac{1}{\zeta} \left(\frac{\lambda}{\zeta} - \rho u \right) - \frac{\gamma}{\zeta^2} \,\frac{\tanh\left(\frac{\gamma}{2} \, t\right) + \eta}{1 + \eta \, \tanh\left(\frac{\gamma}{2} \, t\right)}  \\
\phi(t,u,w) & = \mu\,\frac{\lambda}{\zeta}\left(\frac{\lambda}{\zeta} - \rho u \right) t - 2\mu\frac{\lambda}{\zeta^2}\,\log\!\left(\cosh\left(\frac{\gamma}{2} t\right) + \eta \,\sinh\left(\frac{\gamma}{2} t\right)\right) \!+ t \tilde{\kappa}(u) \:,
\end{aligned}
\end{equation}
where $\gamma = \gamma(u) = \zeta \, \sqrt{\left(\frac{\lambda}{\zeta} - \rho u\right)^2 \!\!+ \frac{1}{4} - \!\left(u-\frac{1}{2}\right)^2}$ and $\eta = \eta(u,w) = \frac{\lambda - \zeta\rho u - \zeta^2 w}{\gamma(u)}$. 
Furthermore, for the Heston model with negative jumps, the function $h$ is given by
\begin{equation}\label{eq:HestonWithJumpsH}
h(u) = \mu\,\frac{\lambda}{\zeta}\left(\frac{\lambda}{\zeta} - \rho u \right) - \mu\,\frac{\lambda}{\zeta^2} \, \gamma(u) + \tilde{\kappa}(u) \:.
\end{equation}

Let us now study the effect of the Esscher transform on the dynamics of the Heston model with jumps.
\begin{proposition}\label{prop:HestonWithJumpsDynamicsPTheta}
Let $\P_\theta$ be the measure given by 
\[
\frac{d\P_\theta}{d\P} = \frac{e^{\theta \, X_{T}}}{\E\left[e^{\theta \, X_{T}}\right]} \:.
\] 
Under $\P_\theta$, the dynamics of the $\P$-Heston process with jumps  $(X_t,V_t)$ becomes
\begin{equation}\label{eq:HestonWithJumpsDynamicsPTheta}
\begin{aligned}
dX_t & = \delta dt + \left(\theta + \zeta \rho \,\psi\left(T-t, \theta, 0 \right) -\frac{1}{2} \right) V_t \, dt + \sqrt{V_t} \, d\tilde{W}_t^1 + dJ_t \:,  && X_0 = 0\\
dV_t & = \tilde{\lambda}_t\, ( \tilde{\mu}_t - V_t) \, dt + \zeta \sqrt{V_t} \, d\tilde{W}_t^2 \:,  && V_0 = V_0 \\
& \hspace*{-.5cm} d\q{\tilde{W}^1,\tilde{W}^2}_t = \rho \, dt \:,
\end{aligned}
\end{equation}
where $\tilde{W}$ is 2-dimensional correlated $\P_\theta$-Brownian motion, $\phi$ and $\psi$ are given in \eqref{eq:GRiccatiSolution},
\[
\tilde{\lambda}_t = \lambda - \zeta \theta \rho - \zeta^2 \,\psi\left(T-t, \theta, 0 \right) \quad \text{and} \quad \tilde{\mu}_t = \frac{\lambda \mu}{\tilde{\lambda}_t} 
\]
and $(J_t)_{t \ge 0}$ is a compound Poisson process with jump rate $\frac{r \alpha}{\alpha+\theta}$ and jump distribution $\text{Neg-Exp}(\alpha+\theta)$ under $\P_\theta$. 
\end{proposition}
\begin{proof}
Denote
\[
D(t,X_t,V_t) = \left.\frac{d\P_\theta}{d\P}\right|_{\F_t} = \frac{e^{\phi\left(T-t, \theta, 0 \right)} }{e^{\phi\left(T, \theta, 0 \right) + \psi\left(T, \theta, 0 \right) \, V_0 }} \: e^{ \psi\left(T-t, \theta, 0 \right) \, V_t + \theta \, X_{t} } \:.
\]
The dynamics of $D(t,X_t,V_t)$ can then be expressed using It\=o's Lemma as
\begin{align*}
dD(t,X_t,V_t) 
& = D(t,X_t,V_t) \left(\theta dX_t + \psi\left(T-t, \theta, 0 \right) dV_t\right) + ...\, dt\\
& = D(t,X_t,V_t) \!\left[\sqrt{V_t}\left(\theta dW_t^1 \!+\! \zeta \, \psi\left(T\!\!-\!t, \theta, 0 \right) dW_t^2\right) \!+\! \theta\,(\delta dt \!+\! dJ_t) \right]
\end{align*}
and Girsanov's theorem then shows that 
\[
d\left(\begin{array}{c}
\tilde{W}_t^1 \\
\tilde{W}_t^2
\end{array}\right)
=
d\left(\begin{array}{c}
W_t^1 \\
W_t^2
\end{array}\right) 
- \sqrt{V_t}
\left(\begin{array}{c}
\theta + \zeta \rho \,\psi\left(T-t, \theta, 0 \right) \\
\theta \rho + \zeta \,\psi\left(T-t, \theta, 0 \right)
\end{array}\right) dt 
\]
is a 2-dimensional Brownian motion under the measure $P_\theta$. Replacing $W$ in eq. \eqref{eq:HestonDynamicsP} by $\tilde{W}$ gives eq. \eqref{eq:HestonWithJumpsDynamicsPTheta}. In order to finish the proof, it remains to show that the jump process $(J_t)_{t \ge 0}$ has the desired distribution under $\P_\theta$. Let us calculate the $\P_\theta$-Laplace transform of $J_t$:
\begin{align*}
\E^{\P_\theta}\left[e^{u J_t}\right] 
& = \frac{\E\left[e^{u J_t}\,\E\left[e^{\theta X_T}\,\middle|\, \F_t\right]\right]}{\E\left[e^{\theta X_T}\right]} \\
& = \frac{e^{\phi\left(T-t, \theta, 0 \right)}}{\E\left[e^{\theta X_T}\right]} \E\left[e^{u J_t+\psi\left(T-t, \theta, 0 \right) \, V_t + \theta X_t}\right] \:.
\end{align*}
By independence of the jumps, 
\[
\E\left[e^{u J_t+\psi\left(T-t, \theta, 0 \right) \, V_t + \theta X_t}\right]
= 
e^{\theta\delta \, t} \, \E\left[e^{(u+\theta) J_t}\right] \, \E\left[e^{\psi\left(T-t, \theta, 0 \right) \, V_t + \theta (X_t - \delta \, t -J_t)  }\right] \:,
\]
where $\E\left[e^{(u+\theta) J_t}\right] = e^{-rt \frac{u+\theta}{u + \theta + \alpha}}$. Furthermore, $(X_t - \delta \, t -J_t,\, V_t)_{t \ge 0}$ is a standard Heston process without jumps. Therefore comparing \eqref{eq:RiccatiSolution} and \eqref{eq:GRiccatiSolution}, we find that 
\[
\E\left[e^{\psi\left(T-t, \theta, 0 \right) \, V_t + \theta (X_t - \delta \, t -J_t)  }\right] 
= e^{\phi(t,\theta,\psi\left(T-t, \theta, 0 \right)) - t \frac{r \theta(\theta-1)}{(\alpha+1)(\alpha+\theta)} + \psi(t,\theta,\psi\left(T-t, \theta, 0 \right)) \, V_0} \:.
\]
Using the fact that $\psi(t,\theta,\psi\left(T-t, \theta, 0 \right)) = \psi\left(T, \theta, 0 \right)$ 
and 
\[
\phi\left(T-t, \theta, 0 \right) + \phi(t,\theta,\psi\left(T-t, \theta, 0 \right)) = \phi\left(T, \theta, 0 \right)
\] 
(see eq. (2.1) in \cite{Kel2011}), we finally obtain
\begin{align*}
\E^{\P_\theta}\left[e^{u J_t}\right] 
& = e^{\theta\delta \, t - rt \frac{u+\theta}{u + \theta + \alpha} - t \frac{r \theta(\theta-1)}{(\alpha+1)(\alpha+\theta)}} \\ 
& = e^{\theta\frac{r}{\alpha+1} \, t - rt \frac{u+\theta}{u + \theta + \alpha} - t \frac{r \theta(\theta-1)}{(\alpha+1)(\alpha+\theta)}} = e^{- \frac{r\alpha}{\alpha+\theta} \, t \, \frac{u}{u + (\alpha + \theta)}}\:,
\end{align*}
which is indeed the Laplace transform of a compound Poisson process with jump rate $\frac{r\alpha}{\alpha+\theta}$ and $\text{Neg-Exp}(\alpha+\theta)$-distributed jumps.
\end{proof}

\subsubsection{Numerical results for the European put option}
Similarly to the case of the Heston model without jumps, denoting $\theta=\theta(\{T\})$, we have 
\begin{equation}\label{eq:HHatHestonWithJumpsEuro}
\begin{aligned}
& \quad \hat{H}(\theta) + \int_0^T h(\theta([t,T])) \,dt \\
& = \: \log\left( \frac{K}{1 - \theta} \right) - \theta \log\left(\frac{-\theta\,K/S_0}{1 - \theta}\right) + T \,\mu\,\frac{\lambda}{\zeta} \left(\frac{\lambda}{\zeta} - \rho \theta  - \frac{\gamma(\theta)}{\zeta} \, \right) + T \, \tilde{\kappa}(\theta)
\end{aligned}
\end{equation}
and we obtain the asymptotically optimal $\theta$ by differentiating \eqref{eq:HHatHestonWithJumpsEuro} with respect to $\theta$ and equating the derivative to 0 by dichotomy . \\
We simulate $N=10000$ trajectories of the Heston model with jumps with parameters $\lambda = 1.1$, $\mu = 0.7$, $\zeta = 0.3$, $\rho = -0.5$, $r=2$, $\alpha=3$ and initial values $V_0 = 1.3$ and $S_0 = 1$, under both $\P$, eq. \eqref{eq:HestonWithJumpsDynamicsP}, and $\P_\theta$, eq. \eqref{eq:HestonWithJumpsDynamicsPTheta}, using a standard Euler scheme with 200 discretization steps. For the $\P$-realisations $X^{(i)}$, we calculate the standard Monte-Carlo estimator of the European put price and for the $\P_\theta$-realisations $X^{(i,\theta)}$, we use \eqref{eq:EstimatorTheta} where $\phi$ and $\psi$ are given in \eqref{eq:GRiccatiSolution} and compute the same statistics as in the previous examples.
In Table \ref{tab:PutVarianceReductionMaturityJumps}, we fix the strike to the value $K=1$ and let the maturity $T$ vary from $0.25$ to $3$, whereas in Tables \ref{tab:PutVarianceReductionStrikeJumpT1} and \ref{tab:PutVarianceReductionStrikeJumpT3}, we fix the maturity to $T=1$ and to $T=3$, while we let the strike $K$ vary between $0.25$ and $1.75$.
\begin{table}[H]
\centering
\begin{tabular}{cccccc}
\hline
$T$ & Price & Std. error & Var. ratio & Adj. ratio & Time, s \\
\hline
\hline
0.25 & 0.0945 & 9.96 $\cdot 10^{-4}$ & 3.28 & 3.00 & 23.6 \\
0.5 & 0.147 & 1.28 $\cdot 10^{-3}$ & 3.20 & 2.99 & 24.5 \\
1 & 0.215 & 1.61 $\cdot 10^{-3}$ & 2.95 & 2.77 & 24.7 \\
2 & 0.309 & 2.04 $\cdot 10^{-3}$ & 2.61 & 2.43 & 24.7 \\
3 & 0.374 & 2.30 $\cdot 10^{-3}$ & 2.40 & 2.20 & 25.0 \\
\hline
\end{tabular}
\caption{The variance ratio as function of the maturity for the European put option in the Heston model with jumps.}
\label{tab:PutVarianceReductionMaturityJumps}
\end{table}
\begin{table}[H]
\centering
\begin{tabular}{cccccc}
\hline
$K$ & Price & Std. error & Var. ratio & Adj. ratio & Time, s \\
\hline
\hline
0.25 & 0.00606 & 7.83 $\cdot 10^{-5}$ & 11.6 & 10.4 & 25.8 \\
0.5 & 0.0377 & 4.03 $\cdot 10^{-4}$ & 5.42 & 5.28 & 24.7 \\
0.75 & 0.105 & 9.44 $\cdot 10^{-4}$ & 3.76 & 3.19 & 27.3 \\
1 & 0.215 & 1.61 $\cdot 10^{-3}$ & 2.93 & 2.89 & 26.1 \\
1.25 & 0.369 & 2.26 $\cdot 10^{-3}$ & 2.65 & 2.46 & 25.4 \\
1.5 & 0.550 & 2.80 $\cdot 10^{-3}$ & 2.43 & 2.24 & 24.9 \\
1.75 & 0.766 & 3.05 $\cdot 10^{-3}$ & 2.57 & 2.44 & 24.6 \\
\hline
\end{tabular}
\caption{The variance ratio as function of the strike for the European put option with maturity $T=1$ in the Heston model with jumps.}
\label{tab:PutVarianceReductionStrikeJumpT1}
\end{table}
\begin{table}[H]
\centering
\begin{tabular}{cccccc}
\hline
$K$ & Price & Std. error & Var. ratio & Adj. ratio & Time, s \\
\hline
\hline
0.25 & 0.0280 & 2.69 $\cdot 10^{-4}$ & 5.19 & 4.99 & 24.8 \\
0.5 & 0.108 & 8.60 $\cdot 10^{-4}$ & 3.32 & 3.05 & 25.1 \\
0.75 & 0.226 & 1.58 $\cdot 10^{-3}$ & 2.68 & 2.56 & 26.3 \\
1 & 0.374 & 2.31 $\cdot 10^{-3}$ & 2.39 & 2.20 & 27.0 \\
1.25 & 0.545 & 3.01 $\cdot 10^{-3}$ & 2.20 & 2.19 & 25.2 \\
1.5 & 0.730 & 3.66 $\cdot 10^{-3}$ & 2.09 & 1.94 & 24.6 \\
1.75 & 0.932 & 4.27 $\cdot 10^{-3}$ & 1.97 & 1.83 & 24.8 \\
\hline
\end{tabular}
\caption{The variance ratio as function of the strike for the European put option with maturity $T=3$ in the Heston model with jumps.}
\label{tab:PutVarianceReductionStrikeJumpT3}
\end{table}
When adding negative jumps to the Heston model, one can see that the variance ratio diminishes. When the options are out of the money however it is still sufficiently important to make it interesting to use in applications. In Figure \ref{fig:VarianceHestonJumps}, we fix the maturity to $T=1.5$ and plot again the empirical variance of the estimator \eqref{eq:EstimatorTheta} as a function of $\theta$ for the Heston model with jumps. The method provides $\theta = -0.312$ as asymptotically optimal measure change which is, as in the continuous case, very close to the optimal one.
\begin{figure}[H]
\centering
\includegraphics[width=10cm,keepaspectratio=true]{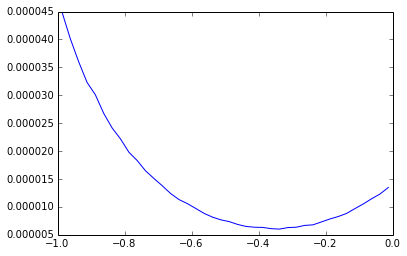}
\caption{The variance of the Monte-Carlo estimator as a function of $\theta$ for the Heston model with jumps.}
\label{fig:VarianceHestonJumps}
\end{figure}

\vspace{\fill}
\bibliographystyle{apalike}
\bibliography{biblio}

\end{document}